\documentclass{amsart}
\usepackage{amssymb}
\usepackage[matrix,arrow,curve]{xy}
\usepackage{verbatim}
\usepackage{a4wide}
\usepackage[applemac]{inputenc}
\usepackage{enumitem}

\newtheorem{theorem}{Theorem}[section]
\newtheorem{lemma}[theorem]{Lemma}
\newtheorem{question}[theorem]{Question}

\newtheorem{definition}[theorem]{Definition}
\newtheorem{conjecture}[theorem]{Conjecture}
\newtheorem{corollary}[theorem]{Corollary}

\newtheorem{proposition}[theorem]{Proposition}

\def\eps{\varepsilon}
\def\om{\omega}

\def\GL{\mathrm{GL}}
\def\SL{\mathrm{SL}}
\def\PSL{\mathrm{PSL}}

\def\C{\mathbf{C}}
\def\Z{\mathbf{Z}}

\def\kk{\mathbf{k}}
\def\F{\mathbf{F}}
\def\Q{\mathbf{Q}}
\def\C{\mathbf{C}}
\def\R{\mathbf{R}}
\def\un{\mathbf{1}}

\def\End{\mathrm{End}}
\def\Ker{\mathrm{Ker}\,}
\def\rk{\mathrm{rk}}
\def\Imm{\mathrm{Im}}
\def\Hom{\mathrm{Hom}}

\def\onto{\twoheadrightarrow}

\def\Id{\mathrm{Id}\,}

\def\ii{\mathrm{i}}
\def\la{\lambda}

\date{November 14, 2016.}

\usepackage{tikz}
\newcounter{braid}
\newcounter{strands}
\pgfkeyssetvalue{/tikz/braid height}{1cm}
\pgfkeyssetvalue{/tikz/braid width}{1cm}
\pgfkeyssetvalue{/tikz/braid start}{(0,0)}
\pgfkeyssetvalue{/tikz/braid colour}{black}
\pgfkeys{/tikz/strands/.code={\setcounter{strands}{#1}}}

\makeatletter
\def\cross{%
  \@ifnextchar^{\message{Got sup}\cross@sup}{\cross@sub}}

\def\cross@sup^#1_#2{\render@cross{#2}{#1}}

\def\cross@sub_#1{\@ifnextchar^{\cross@@sub{#1}}{\render@cross{#1}{1}}}

\def\cross@@sub#1^#2{\render@cross{#1}{#2}}

\def\render@cross#1#2{
  \def\strand{#1}
  \def\crossing{#2}
  \pgfmathsetmacro{\cross@y}{-\value{braid}*\braid@h}
  \pgfmathtruncatemacro{\nextstrand}{#1+1}
  \foreach \thread in {1,...,\value{strands}}
  {
    \pgfmathsetmacro{\strand@x}{\thread * \braid@w}
    \ifnum\thread=\strand
    \pgfmathsetmacro{\over@x}{\strand * \braid@w + .5*(1 - \crossing) * \braid@w}
    \pgfmathsetmacro{\under@x}{\strand * \braid@w + .5*(1 + \crossing) * \braid@w}
    \draw[braid] \pgfkeysvalueof{/tikz/braid start} +(\under@x pt,\cross@y pt) to[out=-90,in=90] +(\over@x pt,\cross@y pt -\braid@h);
    \draw[braid] \pgfkeysvalueof{/tikz/braid start} +(\over@x pt,\cross@y pt) to[out=-90,in=90] +(\under@x pt,\cross@y pt -\braid@h);
    \else
    \ifnum\thread=\nextstrand
    \else
     \draw[braid] \pgfkeysvalueof{/tikz/braid start} ++(\strand@x pt,\cross@y pt) -- ++(0,-\braid@h);
    \fi
   \fi
  }
  \stepcounter{braid}
}

\tikzset{braid/.style={double=\pgfkeysvalueof{/tikz/braid colour},double distance=1pt,line width=2pt,white}}

\newcommand{\braid}[2][]{%
  \begingroup
  \pgfkeys{/tikz/strands=2}
  \tikzset{#1}
  \pgfkeysgetvalue{/tikz/braid width}{\braid@w}
  \pgfkeysgetvalue{/tikz/braid height}{\braid@h}
  \setcounter{braid}{0}
  \let\dsigma=\cross
  #2
  \endgroup
}
\makeatother
\begin{document}
\centerline{}

\title{Artin groups and Yokonuma-Hecke algebras}
\author[I.~Marin]{Ivan Marin}
\address{LAMFA, Universit\'e de Picardie-Jules Verne, Amiens, France}
\email{ivan.marin@u-picardie.fr}
%\subjclass[2010]{20F36, 20F55, 20H15}
\medskip

\begin{abstract} 
We attach to every Coxeter system $(W,S)$ 
an extension $\mathcal{C}_W$ of the corresponding Iwahori-Hecke
algebra. 
We construct a 1-parameter family of (generically surjective) morphisms from
the group algebra of the corresponding Artin group onto $\mathcal{C}_W$. When $W$ is finite,
we prove that this algebra is a free module of finite rank which is generically semisimple. When $W$ is the Weyl group of a Chevalley group, $\mathcal{C}_W$ naturally maps to the associated Yokonuma-Hecke algebra.  When $W = \mathfrak{S}_n$ this algebra can be identified with a diagram algebra called the algebra of `braids and ties'. The image of the usual braid group in this case is investigated. Finally,
we generalize our construction to finite complex reflection groups, thus extending the Brou\'e-Malle-Rouquier construction of a generalized Hecke algebra attached to these groups. 
\end{abstract}

\maketitle

\tableofcontents

\section{Introduction}

Let $(W,S)$ be a Coxeter system, as in \cite{BOURB456}, and $m_{st} \in \Z_{\geq 2} \cup \{ \infty\}$ denote the order of $st$ for $s,t \in S$. Three objects are classically attached
to it : another group, the Artin group $B$ defined by a presentation made of `braid relations',
$$
\langle S \ | \underbrace{sts\dots}_{m_{st}} = \underbrace{tst\dots}_{m_{st}}\ \ \ \forall  s,t \in S \rangle,
$$
a monoid $B^+$ of positive braids defined by the same presentation, and
an algebra, called the Iwahori-Hecke algebra. This algebra $H_W$ 
is defined over a ring $\kk$ containing elements $u_s,s\in S$ subject to the condition $u_s = u_t$ if $s,t$ both lie in the same
conjugacy class, as the quotient of the monoid algebra $\kk B^+$ by the relations
$(s-1)(s+u_s) = 0$ for $s \in S$. It is a deformation of the group algebra of $W$, obtained by the
specialization at $u_s=1$. When $W$ is the Weyl group of some reductive group, $H_W$
admits a natural interpretation as a convolution algebra. The specialization at $u_s = -1$ of $H_W$
admits a natural central extension which is also a quotient of $\kk B$, recently defined in \cite{TRBMW}.

In this paper we define another natural object,
a $\kk$-algebra $\mathcal{C}_W$ which is an
extension of $H_W$, and admits a 1-parameter family of morphisms $B \to \mathcal{C}_W$.
This algebra admits generators $g_s,e_s, s \in S$ and is defined by generators  and relations in section
\ref{sect:genconstruction}.
We prove (see theorem \ref{thm:main}) that it is a free $\kk$-module.
When $W$ is finite, we show that $\mathcal{C}_W$
has rank $|W|.Bell(W)$, where $Bell(W)$ is a natural generalization of the Bell number
$Bell_n$ of partitions of a set of $n$ elements, namely the number of reflection subgroups of $W$.
Precisely, in the general case a basis of $\mathcal{C}_W$ is naturally indexed by couples $(w,W_0)$ for $w \in W$ and $W_0$ a finitely-generated reflection subgroup of $W$.

The original motivation for this algebra comes from an analysis of the so-called Yokonuma-Hecke
algebra associated to a Chevalley group $G$ and its unipotent radical $\mathfrak{U}$,
namely the Hecke convolution ring $\mathcal{H}(G,\mathfrak{U})$, defined by Yokonuma in \cite{YOKO}. Assume $W$ is the Weyl group of $G$,
with generating set $S$. Part of the natural generators of this algebra
are directly connected to the structure of the torus, while the other ones
are in 1-1 correspondence with $S$ and satisfy braid relations, together with a quadratic relation also involving elements of the torus.
In \cite{JUYUNOUVEAUX}, using a Fourier transform construction, J. Juyumaya introduced other natural `braid' generators $g_s, s \in S$, for which
the quadratic relation now involves some idempotent $e_s$ (in which is `hidden' a linear
combinations of elements of the torus). Therefore, there is a natural subalgebra
generated by the $g_s,e_s$, and a natural question is to find a presentation for this
subalgebra, at least when the field of definition of $G$ is generic enough.
The algebra $\mathcal{C}_W$ that we introduce provides an answer to that question. More precisely,
a better answer is a natural quotient $\mathcal{C}_W^R$ of $\mathcal{C}_W$ where reflection subgroups,
in natural 1-1 correspondence with root subsystems, are identified if they
have the same closure (see section \ref{sect:meaning}).

Although one is, at least since Tits's classical article \cite{TITS}, somewhat accustomed to such a phenomenon, it remains
surprising that once again such an object arising from reductive groups admits a natural generalization
to arbitrary Coxeter groups. This algebra $\mathcal{C}_W$ can be viewed as a deformation
of the semidirect product $\mathcal{C}_W(1)$ of the group algebra of $W$ with a commutative algebra spanned by the collection of finitely generated reflection subgroups of $W$. We show in theorem \ref{thm:semisimple} that,
when $W$ is finite and under obvious conditions on the characteristic, this algebra $\mathcal{C}_W(1)$ is semisimple, and therefore $\mathcal{C}_W$ is generically semisimple. For $W = \mathfrak{S}_n$ this generalizes and provides a more
direct proof of a result of \cite{BANJO}. Actually, we show that in the case $W = \mathfrak{S}_n$ and in characteristic $0$, the algebra $\mathcal{C}_W(1)$ is \emph{split} semisimple. 
The question about a similar statement for other Weyl groups raises new problems on the normalizers of reflection and parabolic subgroups in finite Weyl groups
(see section  \ref{sect:spectss}).

In section \ref{sect:braidimage} we introduce a family of morphisms $\Psi_{\underline{\la}} : \kk B \to \mathcal{C}_W(\underline{u})$ and
we exhibit an unexpected connection between the quotient of the group algebra of the braid group appearing inside the Yokonuma-Hecke algebra of type $A$ and (a specialization of) the one connected
with the Links-Gould polynomial invariant of knots and links. We are then able to deduce from Ishii's work on the Links-Gould invariant a new relation inside the Yokonoma-Hecke algebra. Amusingly enough, we notice that Ishii's work and Juyumaya's work on these previously unrelated topics appeared following each other in the same issue of the same
journal (see \cite{JUYU,ISHII}).

A natural question is 
whether the natural map $B \to \mathcal{C}_W(\underline{u})$
is injective. 
Since there is a natural (surjective) map $\mathcal{C}_W(\underline{u})
\to H_W(\underline{u})$, this would be the case if the induced map $B \to
 H_W(\underline{u})$ was itself injective. 
 Right know, this is an open question,
 settled (positively) only in rank 2, by work of Squier \cite{SQUIER},
and an alternative proof can be found in \cite{LEHXI}. 
Our question
of whether $B \to \mathcal{C}_W(\underline{u})$ is injective therefore may or may not be a consequence of the solution of this one. 
A possibly easier question is whether the (restriction to $B$ of the)  maps $\Psi_{\underline{\la}}$ are injective
for generic $\underline{\la}$. We show in section \ref{monoidrepresentation} that a simpleminded application of the existing methods does not suffice to conclude on this point. They however incite to define and look
at a new \emph{monoid} representation $B^+ \to \mathcal{C}_W$ with positive coefficients.

In the last section, we show that the natural quotient $\mathcal{C}_W^p$ of $\mathcal{C}_W$,
where reflection subgroups are identified if they have the same parabolic closure,
can be generalized to the setting where $W$ is a finite complex reflection group, in such
a way that $\mathcal{C}_W^p$ is a natural extension of the generalized Hecke algebra
$H_W$ associated to $W$ by Brou\'e, Malle and Rouquier in \cite{BMR}. The main conjecture
on $H_W$, that $H_W$ is a free module of finite rank, is naturally extended to
an a priori stronger conjecture on $\mathcal{C}_W^p$, that we prove to be true for
a couple of cases. In particular we prove this conjecture for $W$ the complex reflection group of
monomial $n \times n$ matrices with coefficients $d$-th roots of $1$, which provides
a natural extension of the so-called Ariki-Koike algebra.

As a conclusion, we wonder whether other classical objects attached to Iwahori-Hecke algebras, like Kazhdan-Lusztig bases and Soergel bimodules,
can be naturally extended to this setting. In particular it would be interesting to construct an extension of Lusztig's isomorphism of \cite{LUSZ} to $\mathcal{C}_W$.
We also consider very likely that the whole machinery
of Cherednik algebras, including the so-called KZ functor, can be generalized in a natural way to our `extended'
setting. We leave this to future work.

\medskip

{\bf Acknowledgements.} I thank R. Abdellatif, S. Bouc, C. Cornut, T. Gobet, K. Sorlin, R. Stancu and
especially F. Digne and J.-Y. Hée for discussions on root systems and Coxeter groups. I thank A. Esterle
for a careful reading of a first draft. 
\section{Preliminaries}

\subsection{Yokonuma-Hecke algebra}

\label{sect:YHvintage}
Following Yokonuma's original paper \cite{YOKO}, we use Chevalley's notation as in \cite{CHEVALLEY}.
Let $G$ be the Chevalley group associated to a semi-simple complex
Lie algebra $\mathfrak{g}$ and to a finite field $K=\F_q$ and $\mathfrak{H},\mathfrak{W}
,\mathfrak{U}\subset G$ as in \cite{CHEVALLEY}. 
In modern terms, $G$ is a split simple Lie group of adjoint type over $\F_q$,
$\mathfrak{H}$ is a fixed maximal torus, $\mathfrak{U}$ the unipotent
radical of a fixed Borel subgroup containing $\mathfrak{H}$, and $\mathfrak{W}$ is the normalizer of $\mathfrak{H}$ in $G$. In general, the correspondence with modern notations is explained in \cite{CARTER,STEINBERG}. 

To each root $\alpha$ of $\mathfrak{g}$
we let $\varphi_{\alpha} : SL_2(K) \to G$ denote the associated morphism, and
$$
h_{\alpha,t} = \varphi_{\alpha} \left(\begin{array}{cc} t & 0 \\ 0 & t^{-1} \end{array}\right)\ \ 
\om_{\alpha} = \varphi_{\alpha} \left(\begin{array}{cc} 0 & 1 \\ -1 & 0\end{array} \right)\ \ 
$$
Choosing a system $\alpha_1,\dots, \alpha_l$ of simple roots,
we let $\om_i = \om_{\alpha_i}$. There is a short exact sequence $1 \to \mathfrak{H}
\to \mathfrak{W} \to W \to 1$, where $W$ is the corresponding Weyl group.
Each $\om_{\alpha}$ is mapped in $W$ to the reflection $s_{\alpha}$ associated to $\alpha$. The Weyl group admits a presentation as Coxeter system $(W,S)$ with $S = \{ s_1,\dots,
s_l\}$ in 1-1 correspondence with the set of simple roots under $s_i =s_{\alpha_i}\leftrightarrow \alpha_i$. 
The subgroup $\mathfrak{H}$ is generated by the $h_{\alpha,t}$. For short, we let
$h_{i,t} = h_{\alpha_i,t}$. In \cite{CHEVALLEY}, Chevalley denotes $h_{\alpha}$ the
coroot $\check{\alpha}$ associated to $\alpha$. In order to facilitate cross-references
between \cite{CHEVALLEY} and \cite{BOURB456} we will use both notations : $h_{\alpha} = \check{\alpha}$.
The maximal torus $\mathfrak{H}$ is described in \cite{CHEVALLEY}
as the image of $\Hom(L,K^{\times})$, where $L$ is the root lattice,
under the map $\chi \mapsto h(\chi)$ where $h(\chi)$ is an automorphism of the associated complex Lie algebra $\mathfrak{g}$ acting trivially on the Cartan subalgebra and by
$h(\chi)X_r = \chi(r)X_r$ on the generator associated to the root $r$. With these notations,
$h_{\alpha,t} = h(\chi_{\alpha,t})$ where $\chi_{\alpha,t}(r) = t^{r(h_{\alpha})} = t^{ r(\check{\alpha})}$.

In
\cite{YOKO}, th\'eor\`eme 3, T. Yokonuma proves that the Hecke ring $\mathcal{H}(G,\mathfrak{U})$ over $\Z$
admits a presentation by generators $a(h), h \in \mathfrak{H}$,
$a_1,\dots,a_l$ and relations
\begin{enumerate}
\item $a(h_1) a(h_2) = a(h_1h_2)$ for all $h_1,h_2 \in \mathfrak{H}$
\item $a_i a(h) = a(h') a_i$, where $h' = \om_i h \om_i^{-1}$
\item $a_i^2 = q a(h_i) + \sum_{t \in K^{\times}} a(h_{i,t}) a_i$ where $h_i = \om_i^2$
\item $\underbrace{a_i a_j a_i \dots }_{m_{ij}} =\underbrace{a_j a_i a_j \dots }_{m_{ij}}$ for $1 \leq i,j \leq l$
\end{enumerate}

Let $\tilde{e}_i = \sum_{t \in K^{\times}} a(h_{i,t})$ and, in general,
$\tilde{e}_{\alpha} = \sum_{t \in K^{\times}} a(h_{\alpha,t})$. 

The following proposition
is crucial for us. Parts (1) and (2) are standard, parts (3) and (4) appear to be new, at least in
the general case.

\begin{proposition} {\ } \label{prop:relsidsHGU}
\begin{enumerate}
\item For every root $\alpha$, we have $\tilde{e}_{\alpha}^2 = (q-1) \tilde{e}_{\alpha}$,
and $\tilde{e}_{-\alpha} = \tilde{e}_{\alpha}$
\item For every two roots $\alpha, \beta$, we have $\tilde{e}_{\alpha} \tilde{e}_{\beta} = \tilde{e}_{\beta} \tilde{e}_{\alpha}$
\item For every two roots $\alpha, \beta$, we have $\tilde{e}_{\alpha} \tilde{e}_{\beta} = \tilde{e}_{\alpha} \tilde{e}_{s_{\alpha}(\beta)}$
\item For every two roots $\alpha,\beta$, if $\gamma$ is a root such that $\check{\gamma} =  \check{\alpha} + \check{\beta}$, then $\tilde{e}_{\alpha}\tilde{e}_{\gamma} = \tilde{e}_{\alpha} \tilde{e}_{\beta}$.
\end{enumerate}
\end{proposition}
\begin{proof}
We have $\tilde{e}_{-\alpha} = \sum_{t \in K^{\times}} a(h_{-\alpha,t})
%= \sum_{t \in K^{\times}} a(h_{s_{\alpha}(\alpha),t})
= \sum_{t \in K^{\times}} a(h_{s_{\alpha}(\alpha),t})$.
Since, after \cite{CHEVALLEY},
$\{ h_{s_{\alpha}(\alpha),t} ; t \in K^{\times} \}  = \{ \om_{\alpha} h_{\alpha,t} \om_{\alpha}^{-1} ;  t \in K^{\times} \}$
we have $\tilde{e}_{-\alpha} =\sum_{t \in K^{\times}} a(\om_{\alpha} h_{\alpha,t} \om_{\alpha}^{-1})$. From the definitions we have $\om_{\alpha} h_{\alpha,t} \om_{\alpha}^{-1}= \varphi_{\alpha}\left(\begin{array}{cc} -t^{-1} & 0 \\ 0 & -t \end{array}\right)= h_{\alpha,-t^{-1}}$. Since $t \mapsto -t^{-1}$ is a bijection from $K^{\times}$ to $K^{\times}$ this proves $\tilde{e}_{-\alpha} = \tilde{e}_{\alpha}$. Now,
$$
\tilde{e}_{\alpha}^2 = \sum_{t,u \in K^{\times}} a(h_{\alpha,t} h_{\alpha,u})
= \sum_{t,u \in K^{\times}} a(h_{\alpha,tu}) = \sum_{v \in J^{\times}} (\# \{ (t,u) \in K^{\times} ; tu = v\}) a(h_{\alpha,v}) = (q-1) \tilde{e}_{\alpha}.
$$
and this proves (1). Since $\mathfrak{H}$ is commutative (2) is obvious. We now prove (3),
considering two roots $\alpha,\beta$. If $\beta \in \{\alpha,-\alpha \}$  we get the conclusion from $\tilde{e}_{-\alpha} = \tilde{e}_{\alpha}$. Otherwise, $\beta$ and $\alpha$
are linearly independent. 
Then, with obvious notations, $s_{\check{\alpha}}(\check{\beta})$ is the coroot associated to $s_{\alpha}(\beta)$
in the dual root system. By the elementary properties of root systems we
have $s_{\check{\alpha}}(\check{\beta}) = \check{\beta} + m \check{\alpha}$
for some $m \in \Z$. Then, $\chi_{\alpha,t} : x \mapsto t^{x(h_{s_{\alpha}(\beta)})} =
t^{x(h_{\beta})} t^{m x(h_{\alpha})}$ hence

$$
\tilde{e}_{\alpha} \tilde{e}_{s_{\alpha}(\beta)} = \sum_{t,u \in K^{\times}} a(h_{\alpha,t}
h_{s_{\alpha}(\beta),u})
$$
and by definition (see \cite{CHEVALLEY}) $h_{\alpha,t}h_{s_{\alpha}(\beta),u}$ corresponds 
to the element of $\Hom_{\Z}(L,K^{\times})$
which is given by 
$x \mapsto t^{x(h_{\alpha})} u^{x(h_{\beta})+mx(h_{\alpha})} = (tu^m)^{x(h_{\alpha})} u^{x(h_{\beta})}$.
Therefore $h_{\alpha,t}h_{s_{\alpha}(\beta),u} = h_{\alpha,(tu)^m} h_{\beta,u}$
and

$$
\tilde{e}_{\alpha} \tilde{e}_{s_{\alpha}(\beta)} = \sum_{t,u \in K^{\times}} a(h_{\alpha,(tu)^m})
a(h_{\beta,u} ) = \sum_{t,u \in K^{\times}} a(h_{\alpha,t})
a(h_{\beta,u} ) = \tilde{e}_{\alpha} \tilde{e}_{\beta}
$$
since $(t,u) \mapsto (tu^m,u)$ is a bijection from $(K^{\times})^2$ to itself. This proves (3).
The proof of (4) is similar : we get $\tilde{e}_{\alpha} \tilde{e}_{\gamma}
= \sum_{t,u \in K^{\times} } a(h_{\alpha,t} h_{\gamma,u})$ and
$h_{\alpha,t} h_{\gamma,u}$ corrresponds to $x \mapsto t^{x(h_{\alpha})} u^{x(h_{\alpha}+
h_{\beta})} = (tu)^{x(h_{\alpha})} u^{h_{\beta}}$ and we conclude as before. This
proves the claim.
\end{proof}

The maximal torus $\mathfrak{H}$ can be identified with $(K^{\times})^{l}$
through the identification with
$\Hom(L,K^{\times}) = \Hom(\bigoplus_{i=1}^{l} \Z \alpha_i,K^{\times})
= \prod_{i=1}^l \Hom(\Z\alpha_i,K^{\times}) \simeq (K^{\times})^{l}$.
If $\beta_1,\dots,\beta_k$ are roots, and $t_1,\dots,t_k \in K^{\times}$,then
$a(h_{\beta_1,t_1})a(h_{\beta_2,t_2})\dots a(h_{\beta_k,t_k}) \in \mathfrak{H}$
is identified with the $l$-tuple
$$(t_1^{\alpha_i(\check{\beta}_1)}t_2^{\alpha_i(\check{\beta}_2)}
\dots
t_k^{\alpha_i(\check{\beta}_k)})_{1 \leq i \leq l} \in (K^{\times})^l.
$$
Choosing a generator $\zeta$ of $K^{\times}$, and therefore an isomorphism $K^{\times}
\simeq \Z/(q-1)\Z$, it is identified with the $l$-tuple
$$
\left(\zeta^{\alpha_i(m_1 \check{\beta}_1+\dots + m_k  \check{\beta}_k)}\right)_{1 \leq i \leq l} \in (K^{\times})^l
$$
where $t_j = \zeta^{m_j}, m_j \in \Z/(q-1)\Z$,
and therefore with the $l$-tuple $(\alpha_i(m_1  \check{\beta}_1+\dots + m_k  \check{\beta}_k))_{1 \leq i \leq l} \in (\Z/(q-1)\Z)^l$. Let us now assume that $\beta_1,\dots,\beta_k$ forms
a basis of a root subsystem. Then
$\tilde{e}_{\beta_1}\dots \tilde{e}_{\beta_k}$ is mapped inside $\Z \mathfrak{H} \simeq \Z[(\Z/(q-1)\Z)^l]$
to $\sum_{m_1,\dots m_k \in \Z/(q-1)\Z}  [\alpha_i(m_1  \check{\beta}_1+\dots + m_k  \check{\beta}_k)]_{1\leq
i \leq l} $.
We consider the map $\Phi : (\Z/(q-1)\Z)^{k} \to (\Z/(q-1)\Z)^{l}$
given by $(m_1,\dots, m_k) \mapsto [\alpha_i(m_1  \check{\beta}_1+\dots + m_k  \check{\beta}_k)]_{1\leq
i \leq l}$. It is a $\Z$-module homomorphism, with kernel
the set of $m_1,\dots,m_k$ such that $m_1 \check{\beta}_1 + \dots + m_k  \check{\beta}_k$
lies in the kernel of all $\alpha_i$'s modulo $q-1$. 

Therefore $\tilde{e}_{\beta_1}\dots \tilde{e}_{\beta_k}$
is mapped to 
$$
%(\# \Ker \Phi) \sum_{\underline{v} \in \Imm(\Phi)} [\alpha_i(\underline{v})]_{1 \leq i \leq l}.
(\# \Ker \Phi) \sum_{\underline{v} \in \Imm(\Phi)} \underline{v}.
$$
Let $F$ denote the sub-lattice of the co-root lattice spanned by
$\check{\beta}_1,\dots, \check{\beta}_k$, and $C$ the Cartan matrix
of the root system. The values obtained as $\underline{v} \in \Imm(\Phi)$
%$[\alpha_i(\underline{v})]_{1 \leq i \leq l}$
are exactly the image of $F$ under $C$ modulo $q-1$, and $\Ker \Phi$
depends only on $q-1$, $F$ and $C$. Let $r$ be a prime dividing $q-1$ and
not dividing $\det(C)$. We let $\Phi_r : \F_r^k \to \F_r^{l}$ denote
the reduction of $\Phi$ modulo $r$. Then, under the map $\Z[(\Z/(q-1)\Z)^{l}]\to
\Z[\F_r^{l}]$, $\tilde{e}_{\beta_1}\dots \tilde{e}_{\beta_k}$ is mapped
to
$$
(\# \Ker \Phi_r) \sum_{\underline{v} \in \Imm(\Phi_r)} [\alpha_i(\underline{v})]_{1 \leq i \leq l}.
$$
Since $C$ is invertible modulo $r$, the image $\Imm \Phi_r$ of the lattice $F \mod rL$ under $C$
determines $F \mod r$. Since there is a finite number of possible lattices $F$, there
exists $r_0$ such that, for all prime $r \geq r_0$, the knowledge of $F \mod r L$
determines $F$. Let us choose such a prime number. By the Dirichlet prime number theorem there exists a prime $p=q$ such that $p \equiv 1 \mod r$,
that is $r |q-1$. Therefore, the subalgebra generated by the $\tilde{e}_{\alpha}$
is `generically' freely spanned by a family indexed by the collection of all \emph{closed} symmetric
subsystems of (the dual of) our original subsystem. Recall that there exists
reduced root systems with proper closed symmetric subsystems of the same rank,
for instance the long roots in type $G_2$ form a subsystem of type $A_2$ with
this property.

\subsection{Juyumaya's generators}

In \cite{JUYUNOUVEAUX}, Juyumaya introduced new generators $L_i$'s of $\mathcal{H}(G,U)$ in replacement of the $a_i$'s, keeping the $a(h)$ as they are. Choosing a non trivial additive character $\psi$ of $(K,+)$,
and using some kind of Fourier transform,
he defines for every root $\alpha$ the element
$\psi_{\alpha} = \sum_{r \in K^{\times}} \psi(r) h_{\alpha,r}$. Then, letting
$L_{i} = q^{-1}(\tilde{e}_{\alpha_i} + a_{i} \psi_{\alpha_i})$ he shows,
in collaboration with S. Kannan (\cite{JUYUKANNAN}, theorem 2) that
$\mathcal{H}(G,U)$ admits a presentation with generators $L_1,\dots,L_{l}$,
$a(h), h \in \mathfrak{H}$ and relations
\begin{enumerate}
\item $a(h_1) a(h_2) = a(h_1h_2)$ for all $h_1,h_2 \in \mathfrak{H}$
\item $L_i a(h) = a(h')L_i$ where $h' = \omega_i h \omega_i^{-1}$
\item $L_i^2 = 1 - q^{-1}(\tilde{e}_{\alpha_i} - L_i \tilde{e}_{\alpha_i})$
\item $\underbrace{L_i L_j L_i \dots }_{m_{ij}} =\underbrace{L_j L_i L_j \dots }_{m_{ij}}$ for $1 \leq i,j \leq l$
\end{enumerate}
Then, letting $u = q^{-1}$, $e_{\alpha} = (q-1)^{-1} \tilde{e}_{\alpha}$, $e_i = e_{\alpha_i}$
and $g_i = -L_i$, this presentation becomes the following one :
\begin{enumerate}
\item $a(h_1) a(h_2) = a(h_1h_2)$ for all $h_1,h_2 \in \mathfrak{H}$
\item $g_i a(h) = a(h')g_i$ where $h' = \omega_i h \omega_i^{-1}$
\item $g_i^2 = 1 +(u-1) e_i(1+g_i)$
\item $\underbrace{g_i g_j g_i \dots }_{m_{ij}} =\underbrace{g_j g_i g_j \dots }_{m_{ij}}$ for $1 \leq i,j \leq l$
\end{enumerate}

\subsection{Yokonuma-Hecke algebras of type $A$}

A particularly studied variation of the above construction
mimics the situation above for the (non-semisimple !) reductive group $\GL_n(K)$ with
$K$ a `field of order $d+1$'. Let us fix
a commutative ring $\kk$  (with $1$), $u \in \kk$, $d \in \Z_{>0}$. We assume that $d$ and $u$ are invertible in $\kk$. 
The literature on the subject, see e.g. \cite{CHLOULAMB},  denotes $Y_{d,n}(u)$ and calls the Yokonuma-Hecke algebra of type $A$ the $\kk$-algebra generated with generators
$g_1,\dots,g_{n-1}$, $t_1,\dots,t_n$ and relations
\begin{enumerate}
\item $g_i g_{i+1} g_i = g_{i+1} g_i g_{i+1}$, $g_i g_j = g_j g_i$ if $|j-i| \geq 2$ (braid relations), 
\item $t_i t_j = t_j t_j$, $g_i t_j = t_{s_i(j)} g_i$ for all $i,j$, where $s_i$ is the transposition $(i,i+1)$;
\item $t_i^d = 1$ for all $i$,
\item $g_i^2 = 1 + (u-1)e_i(1+g_i)$
\end{enumerate}
where, by definition $e_i = e_{i,i+1}$ with
$$
e_{i,j} = \frac{1}{d} \sum_{s=0}^{d-1} t_i^st_{j}^{-s}.
$$
whenever $i \neq j$ and $1 \leq i,j \leq n$.
The elements $g_i$ are invertible, with inverse $g_i^{-1} = g_i + (u^{-1}-1)e_i + (u^{-1}-1)e_i g_i$.
It can be easily proved
 that
the following relations hold :
\begin{enumerate}[resume]
\item $e_{ij} = e_{ji}$ for all $i \neq j$
\item $e_{i,j} e_{k,l} = e_{k,l} e_{i,j}$ for all $i\neq j$, $k \neq l$
\item $g_i e_{j,k}  = e_{s_i(j),s_i(k)} g_i$ for all $i,j,k$ with $k \neq j$
\item $e_{ij}^2= e_{ij}$ for all $i\neq j$.
\end{enumerate}

The subalgebra of $Y_{d,n}(u)$ generated by the $g_i$'s and $e_i$'s
has been investigated in the past years. J. Juyumaya and F. Aicardi have introduced
a diagram algebra $\mathcal{E}_n(u)$ called the algebra of braids and ties,
such that this subalgebra is an homomorphic image of $\mathcal{E}_n(u)$,
this morphism being generically injective (actually already for $d \geq n$, see \cite{ERH}).
%called the algebra of braids and ties by J. Juyumaya and F. Aicardi,
%according to \cite{RYOMHANSEN}, and provided a diagrammatic
%description. 
A Markov trace was subsequently constructed on this algebra of braids and ties, see
\cite{AICARDIJUYU}. This algebra is efficiently studied in \cite{RYOMHANSEN}, where S. Ryom-Hansen
provides a faithful module for it, and uses it to show that the algebra has dimension $n! Bell_n$, where $Bell_n$ is the $n$-th Bell number. Theorem \ref{thm:main} below generalizes this last statement.

Now we notice that, in \cite{CHOULPDA1}, M. Chlouveraki and L. Poulain d'Andecy introduce
other generators $g'_i = g_i + (v^{-1}-1)e_i g_i$, with $u = v^2$. The relation between $g'_i$
and $e_i$ is then $(g'_i)^2 = 1 + (v-v^{-1})e_i g'_i$.
They notice that these generators also satisfy the braid relations. We will
give a general explanation for this phenomenon in section \ref{sect:braidmorphisms}.
 
\section{Construction of the algebra $\mathcal{C}_W$}

\subsection{General construction}
\label{sect:genconstruction}

Here $\kk$ is a commutative ring (with $1$).
Let $W$ denote
 a 
 Coxeter group, with generating set $S$. We let $\mathcal{R} \supset S$
denote its set of reflections. If $W$ is finite this set can be defined as the
geometric reflections of $W$ in its natural representation, and in the general
case this is the set of conjugates of $S$. We denote $\mathcal{P}_f(\mathcal{R})$
the set of all finite subsets of $\mathcal{R}$, and by $\mathcal{P}(\mathcal{R})$
the set of all its subsets. We recall that a reflection subgroup of $W$ is a subgroup
generated by a subset of $\mathcal{R}$. 

We also recall that a Coxeter group $W$ given by the Coxeter system $(W,S)$ is finitely
generated as a group if and only if $S$ is finite. Indeed, if $W = \langle x_1,\dots, x_n \rangle$
for some $x_1,\dots, x_n$, we can write the $x_i$'s as a product of a finite number of elements of
$S$, hence $W$ is equal to its standard parabolic subgroup $(W_X,X)$ for some finite $X \subset S$.
Since $W_X \cap S = X$ (\cite{BOURB456}, IV \S 1 No. 8, corollaire 2) this proves that $S=X$ is finite.

Finally, we recall from Dyer's thesis
the following basic fact, extending a well-known property of finite Coxeter groups to general ones :
\begin{proposition} \label{prop:dyer} (Dyer, PhD thesis, theorem 1.8; see also \cite{DYERART} corollary 3.11 (ii) and Deodhar \cite{DEODHAR}) Let $W_0$ be a reflection subgroup of $W$. Then $W_0$
is a Coxeter group $(W_0,S_0)$ with $S_0 \subset \mathcal{R}$
and $W_0 \cap \mathcal{R} = \mathcal{R}_0$, with $\mathcal{R}_0$
the set of reflections of $(W_0,S_0)$. Moreover, if $W_0$ is generated by $J \subset \mathcal{R}$, then every
element of $\mathcal{R}_0$ is a conjugate inside $W_0$ of an element of $J$.
\end{proposition}

 For every $s \in S$, we choose $u_s \in \kk$
such that $s_1 \sim s_2 \Rightarrow u_{s_1} = u_{s_2}$,
where $a \sim b$ means that $a,b \in S$ lie in the same conjugacy class.
We set $\underline{u} = (u_s)_{s \in S}$ and define
$\mathcal{C}_W(\underline{u})$ to be the associative unital $\kk$-algebra defined by generators
$g_s,s \in S$, $e_{t}, t \in \mathcal{R}$, and relations
\begin{enumerate}
\item $\underbrace{g_s g_t g_s \dots }_{m_{st}} =\underbrace{g_t g_s g_t \dots }_{m_{st}}$ for $s,t \in S$
\item $e_t^2 = e_t$ for all $t \in \mathcal{R}$
\item $e_{t_1}e_{t_2} = e_{t_2}e_{t_1}$ for all $t_1,t_2 \in \mathcal{R}$
\item $e_{t} e_{t_1} = e_t e_{t t_1t^{-1}}$ for all $t,t_1,t_2 \in \mathcal{R}$
\item $g_s e_t = e_{sts} g_s$ for all $s \in S$, $t \in \mathcal{R}$
\item $g_s^2 = 1 + (u_s -1)e_s(1+g_s)$ for all $s \in S$.
\end{enumerate}

Note that $\mathcal{C}_W(\underline{u})$ is actually finitely generated as soon as $S$ is finite,
by the following elementary proposition.
\begin{proposition} The algebra $\mathcal{C}_W(\underline{u})$ is generated by the
$g_s,e_s$ for $s \in S$.
\end{proposition}
\begin{proof}
Let $A$ be the subalgebra of $\mathcal{C}_W(\underline{u})$ generated by the $g_s,e_s$
for $s \in S$. It is sufficient to show that $e_t \in A$ for all $t \in \mathcal{R}$. By definition such a $t$
can be written as $w^{-1}s_0w$ for some $s_0 \in S$ and $w \in W$. Writing $w = s_1\dots s_r$
with $s_1,\dots, s_r \in S$, we need to prove $e_{s_rs_{r-1}\dots s_1 s_0 s_1 \dots s_r} \in A$
for all $s_0,s_1,\dots,s_r \in S$. By induction on $r$ this results from
the relation $g_{s_r} e_{s_{r-1}\dots s_1 s_0 s_1 \dots s_{r-1}} g_{s_r}^{-1} = e_{s_rs_{r-1}\dots s_1 s_0 s_1 \dots s_{r-1}s_r}$.
\end{proof}

For $w \in W$, we let $g_w = g_{s_1}\dots g_{s_r}$ if $s_{1}\dots s_{r}$
is a reduced expression of $w$. Since the $g_s$'s satisfy the
braid relations this does not depend on the
chosen expression by Iwahori-Matsumoto's theorem.

For $J \in \mathcal{P}_f( \mathcal{R})$, we set $e_J = \prod_{t \in J} e_t$. In order to
study these elements we define an equivalence relation $J \sim K$ on $\mathcal{P}_f(\mathcal{R})$
as the equivalence relation generated by the couples $(J, K) \in \mathcal{P}_f(\mathcal{R}) \times \mathcal{P}_f(\mathcal{R})$
such that $J$ contains some $\{ s,t\}$ and $K
= J \cup \{sts \}$. By definition this is the smallest equivalence relation containing such couples.

This equivalence relation can be restated as follows.

\begin{proposition}
Let $J,K \in \mathcal{P}_f(\mathcal{R})$. Then, $J \sim K$ if and only if $\langle J \rangle = \langle K \rangle$.
\end{proposition}
\begin{proof}
%From this we get that this equivalence relation can be restated as follows. 
%Say that $J$ and $K$ are equivalent if $\langle J \rangle = \langle K \rangle$, that
%is if the reflection subgroups generated by $J$ and $K$ are the same.

It is easy to prove that, if $J \sim K$, then $\langle J \rangle = \langle K \rangle$.
Indeed, the relation $J_1 \equiv J_2$ defined by $\langle J_1 \rangle = \langle J_2 \rangle$ is obviously an equivalence relation,
and, if $K = J \cup \{ sts \}$ with $\{ s,t \} \subset J$,  we have $\langle J \rangle = \langle K \rangle$, that is $J \equiv K$. It follows that
the relation $\equiv$ contains all such couples $(J,K)$, hence $J \sim K \Rightarrow J \equiv K$.

Conversely, let us assume $\langle J \rangle = \langle K \rangle$.
If $J= \emptyset$ or $K = \emptyset$ then $\langle J \rangle  = \{ 1 \} = \langle K \rangle$ implies $J = K = \emptyset$ hence $J \sim K$.
Otherwise, let us
set $W_0 = \langle J \rangle = \langle K \rangle$. 
By proposition \ref{prop:dyer} the group $W_0$ is a Coxeter group with generating set $S_0 \subset \mathcal{R}$,
and $W_0 \cap \mathcal{R} = \mathcal{R}_0$ is the set of all conjugates of elements of $S_0$. 
Moreover, since $W_0 = \langle J \rangle$, proposition \ref{prop:dyer} states that every element $\mathcal{R}_0$
is a conjugate (inside $W_0$) of an element of $J$. This applies in particular to the elements of $K \subset W_0 \cap \mathcal{R} = \mathcal{R}_0$.
Therefore, every $x \in K$ can be written as $m s_0 m^{-1}$ for some $s_0 \in J$ and $m \in W_0 = \langle J \rangle$.
Writing $m = s_r s_{r-1} \dots s_1$ for some $s_1,\dots,s_r \in J$ we get $x = 
%every reflection of
%$W_0$ is conjugate to an element of $J$, and also to an element of $K$. In particular
%every element of $K$ is conjugate to an element of $J$ by some element of $\langle J \rangle$.
%Writing such an element as $
s_r s_{r-1} \dots s_1 s_0 s_1 \dots s_{r-1} s_r$
for $s_0,\dots,s_r \in J$. By induction on $r$ one gets readily that $x \in J_x$, for some $J_x \in \mathcal{P}_f(\mathcal{R})$ with $J_x \sim J$ and $J \cup J_x$.
Since $\langle J_x \rangle = \langle J \rangle = \langle K \rangle$ and $K$ is finite, we can iterate this argument for all elements $x \in K$,
and this proves that $K \subset J'$ for some $J' \in \mathcal{P}_f(\mathcal{R})$ with $J \subset J'$ and $J \sim J'$.

Therefore, we can assume $K \subset J$. By the same argument, every
element of $J$ can be written as $s_r s_{r-1} \dots s_1 s_0 s_1 \dots s_{r-1} s_r$
for $s_0,\dots,s_r \in K \subset J$. By induction on $|J \setminus K|$ we get from this that $J \sim K$.
\end{proof}

Therefore, the set of equivalence classes is in natural bijection with the collection $\mathcal{W}$
of finitely generated reflection subgroups of $W$. In particular,
when $W$ is finite, the number of equivalence classes can be identified with the
number of reflection subgroups of $W$. Notice that, when $W$ is the Weyl group
of some root system $R$, then reflection subgroups are in 1-1
correspondence with root subsystems (in the sense of a subset of $R$ satisfying the
axioms of root systems, as in \cite{BOURB456}).

By relations (2) and (4) above, we have $e_s e_t = e_s e_t e_t = e_s e_{sts} e_t$
and thus $J \sim K$ implies $e_J = e_K$. Therefore, we can define $e_{W_0}$
for every finitely generated reflection subgroup $W_0$ of $W$, by letting $e_{W_0} = e_J$
for any $J \in \mathcal{P}_f(\mathcal{R})$ with $\langle J \rangle = W_0$. 
Notice that, when $W$ is finite, there is a distinguished representative of the class of $J \in
\mathcal{P}_f(\mathcal{R}) = \mathcal{P}(\mathcal{R})$, namely $\overline{J} := \langle J \rangle \cap \mathcal{R}$. In the general case, one can make a different choice, taking for $\overline{J}$ Dyer's canonical set of Coxeter generators for $\langle J \rangle$ (since such set can be infinite only
if the Coxeter group is not finitely generated). In the sequel, we will denote $\overline{J} \in \mathcal{P}_f(\mathcal{R})$ the chosen representative of the class of $J \in \mathcal{P}_f(\mathcal{R})$.

\subsection{Description as a module}

\begin{theorem} \label{thm:main} The algebra $\mathcal{C}_W(\underline{u})$ is a free $\kk$-module
with basis the $e_{\bar{J}} g_w$, for $w \in W$ and $J \in \mathcal{P}_f(\mathcal{R})$. In particular,
if $W$ is finite then it has for rank the order $|W|$ of $W$ multiplied by the number $|\mathcal{W}|$
of reflection subgroups of $W$.
%the Bell number of type $W$.
\end{theorem}
We shall see in section \ref{sect:Bell} that $|\mathcal{W}|$ may be called the Bell number of type $W$.
\begin{proof}
We denote by $\ell$ the classical length function on the Coxeter group $W$.
To each $J \in \mathcal{P}_f(\mathcal{R})$ we associate $e_J = \prod_{t \in J} e_t$. Let us consider
$J \in \mathcal{P}_f(\mathcal{R})$, $w \in W$ and $s \in S$. Then $g_se_J g_w = e_{sJs^{-1}} g_s g_w$.
If $\ell(sw) = \ell(w)+1$ we have $g_s g_w = g_{sw}$
and we get $g_s. e_J g_w = e_{sJs^{-1}} g_{sw}$. Otherwise
$w$ can be written $w = s w'$ with $\ell(w') = \ell(w)-1$. Then $g_s g_w = 
g_s^2 g_{w'} = g_{w'} + (u_s-1) e_s (1+g_s) g_{w'} = 
g_{w'} + (u_s-1) e_s g_{w'} + (u_s-1) e_s g_s g_{w'}
= g_{w'} + (u_s-1) e_s g_{w'} + (u_s-1) e_s g_w$. It follows that
$g_s .e_J g_w = e_{sJs}g_{w'} + (u_s-1) e_{sJs}e_s g_{w'} + (u_s-1) e_{sJs}e_s g_w
= e_{sJs}g_{sw} + (u_s-1) e_{sJs\cup \{ s \}} g_{sw} + (u_s-1) e_{sJs \cup \{ s \}} g_w$.
Finally, in all cases we have $e_s. (e_J g_w) = e_{J \cup \{ s \}} g_w$.
Since $\mathcal{C}_W(\underline{u})$ is generated as a unital algebra by the $g_s$ and $e_s$, $s \in S$ this
proves that the set of the $e_J g_w$ for $J \in \mathcal{P}_f(\mathcal{R})$, $w \in W$,
and therefore of the $e_{\bar{J}} g_w$ for $J\in \mathcal{P}_f(\mathcal{R})$, $w \in W$, is
a spanning set for $\mathcal{C}_W(\underline{u})$.

We notice that $(e_J g_w) e_s = e_J e_{w s w^{-1}} g_w = e_{J \cup \{ w s w^{-1} \}} g_w$
and, if $\ell(ws) = \ell(w)+1$, then $(e_J g_w) g_s = e_J g_{ws}$. If $\ell(ws) = \ell(w)-1$,
then $e_J g_w g_s = e_J g_{ws} g_s^2 = e_Jg_{ws}(1+ (u_s-1) e_s(1+g_s))
= e_Jg_{ws} +(u_s-1) e_Jg_{ws} e_s + (u_s-1)e_Jg_{ws}e_s g_s
= e_Jg_{ws} +(u_s-1) e_Je_{ws.s.(ws)^{-1}}g_{ws}  + (u_s-1)e_Je_{ws.s.(ws)^{-1}}g_{ws} g_s
= e_Jg_{ws} +(u_s-1) e_Je_{wsw^{-1}}g_{ws}  + (u_s-1)e_Je_{wsw^{-1}}g_{w}
= e_Jg_{ws} +(u_s-1) e_{J \cup \{ wsw^{-1} \}}g_{ws}  + (u_s-1)e_{J \cup \{ wsw^{-1}\} }g_{w}$.

We now consider a free $\kk$-module $V$ with basis
$v_{J,w}$ for $J \in \mathcal{P}_f(\mathcal{R})$, $w \in W$, with the convention $v_{J,w} = v_{K,w}$ if $J \sim K$.
We introduce $\kk$-linear endomorphisms
$G_s,E_s, G'_s, E'_s  \in \End(V)$ defined by the formulas $E_s.v_{J,w} = v_{J \cup \{ s \}, w}$
and
$$
\begin{array}{lcll}
G_s .v_{J,w} &=&v_{sJs, sw} &\mbox{\ if\ }\ell(sw)  = \ell(w)+1 \\
 &=&v_{sJs,sw} + (u_s-1) v_{sJs\cup \{ s \},sw} + (u_s-1) v_{sJs \cup \{ s \},w} &\mbox{\ if\ }\ell(sw)  = \ell(w)-1 \\
E_s .v_{J,w} &=&v_{J\cup \{ s \}, w} \\
E'_s .v_{J,w} &=&v_{J\cup \{ wsw^{-1} \}, w} \\
G'_s .v_{J,w} &=&v_{J, ws} &\mbox{\ if\ }\ell(ws)  = \ell(w)+1 \\
&=&v_{J,ws} +(u_s-1) v_{J \cup \{ wsw^{-1} \},ws}  + (u_s-1)v_{J \cup \{ wsw^{-1}\} ,w}
&\mbox{\ if\ }\ell(ws)  = \ell(w)-1 \\
 \end{array}
$$
We easily check on these formulas that
$E_s^2 = E_s$, $(E'_s)^2 = E'_s$. Moreover
$$E_s E'_t . v_{J,w} = E_s . v_{J \cup \{ wtw^{-1}\},w} = 
 v_{J \cup \{ wtw^{-1}\} \cup \{ s \},w},$$
 while $E'_t E_s v_{J,w} = E'_t. v_{ J \cup \{s \},w}
 = v_{J \cup \{ s \} \cup \{ w tw^{-1} \},w}$. This proves $E_s  E'_t = E'_t E_s$ for all $s,t$.
 Similarly, if $\ell(wt) = \ell(w)+1$, we have $E_s G'_t . v_{J,w}= E_s v_{J, wt}
 = v_{J \cup \{ s \}, wt}$. Otherwise $G'_t E_s . v_{J,w} = G'_t.v_{J \cup \{ s \},w} 
 = v_{J \cup \{ s \},wt}$ ; if $\ell(wt) = \ell(w)-1$, we have 
 $E_s G'_t . v_{J,w}= E_s .(v_{J,wt} +(u_t-1) v_{J \cup \{ wtw^{-1} \},wt}  + (u_t-1)v_{J \cup \{ wtw^{-1}\} ,w})
= v_{J\cup \{ s \},wt} +(u_t-1) v_{J \cup \{ wtw^{-1} \}\cup \{ s \},wt}  + (u_t-1)v_{J \cup \{ wtw^{-1}\}\cup \{ s \} ,w}$
and $G'_tE_s v_{J,w} = G'_t v_{J \cup \{ s \},w}
= v_{J \cup \{ s \},wt} +(u_t-1) v_{J \cup \{ wtw^{-1} \} \cup \{ s \},wt}  + (u_t-1)v_{J \cup \{ s \} \cup \{ wtw^{-1}\} ,w}$, which proves $G'_t E_s = E_s G'_t$ for all $s,t$.
By a similar computation we get $G_t E'_s = E'_s G_t$ for all $s,t$.

We now want to check that $G_s G'_t = G'_t G_s$. We first recall the following classical fact, of which we recall a proof for the convenience of the reader:
\begin{lemma} Let $(W,S)$ be a Coxeter system. For $s,t \in S$ and $w \in W$, the
equalities $\ell(swt) = \ell(w) $ and $\ell(sw) = \ell(wt)$ imply $sw=wt$.
\end{lemma}
\begin{proof} $\ell(swt) = \ell(w)$ implies that, either $\ell(sw) = \ell(w)+1$ and $\ell(swt) = \ell(sw)-1$,
or $\ell(sw) = \ell(w)-1$ and $\ell(swt) = \ell(sw)+1$. We start dealing with the first case. Let $n = \ell(w)$
and $s_1\dots s_n = w$ a reduced expression. Since $\ell(wt) = \ell(sw) = \ell(w)+1$ we
get that $wt = s_1\dots s_n t$ is again a reduced expression. Since $\ell(s.wt) < \ell(wt)$
we get from the exchange lemma that, either $swt = s_1 \dots s_{j-1} s_{j+1}\dots s_nt$ for some
$j \in \{1,\dots, n\}$, or $swt = s_1\dots s_n$. In the first case we would have $sw = s_1 \dots s_{j-1} s_{j+1}
\dots s_n$, contradicting $\ell(sw) = n+1$. Therefore $swt = s_1\dots s_n = w$ and $sw = wt$.

Now, if $\ell(sw) = \ell(w)-1$ and $\ell(swt) = \ell(sw)+1$, letting $w' = sw$ we can apply
the previous discussion and get $sw' = w't$, that is $sw = wt$.

\end{proof}

If $\ell(sw) = \ell(w)+1$ and $\ell(wt) = \ell(w)+1$ then, either $\ell(swt) = \ell(wt)+1 = 
\ell(sw)+1$, or $\ell(swt) = \ell(wt)-1 = \ell(w)$ in which case $sw = wt$. In the first case,
we have $G_s G'_t v_{J,w} = G_s v_{J ,wt} = v_{sJs,swt}$ and
$G'_t G_s v_{J,w} = G'_tv_{sJs,sw} = v_{sJs,swt}$ ; in the second case,
we have $G_s G'_t v_{J,w} = G_s v_{J,wt} = G_s v_{J,sw} = 
v_{sJs,w} + (u_s-1) 
v_{sJs \cup \{ s \},w} +
(u_s-1) v_{sJs \cup \{ s \},sw}$ and
$G'_t G_s v_{J,w} = G'_t v_{sJs,sw} = G'_t v_{sJs,wt}
= v_{sJs,w} + (u_t-1)v_{sJs \cup \{ w t w^{-1} \},w} + (u_t-1) v_{sJs \cup \{ w t w^{-1} \}, wt}$.
Since the condition $sw = wt$ implies $wtw^{-1} = s$ and in particular $s \sim t$, whence
$u_s = u_t$. Therefore, $G_s G'_t. v_{J,w}= G'_t G_s. v_{J,w}$.

If $\ell(sw) = \ell(w)+1$ and $\ell(wt) = \ell(w)-1$, then
we have $\ell(swt) = \ell(w)$, for otherwise $\ell(swt) = \ell(w)-2$ and $\ell(sw)< \ell(w)$. Then
$G_s G'_t v_{J,w} = 
G_s(v_{J,wt} +(u_t-1) v_{J \cup \{ wtw^{-1} \},wt}  + (u_t-1)v_{J \cup \{ wtw^{-1}\} ,w}) =
G_s(v_{J,wt} +(u_t-1) v_{J \cup \{ wtw^{-1} \},wt}  + (u_t-1)v_{J \cup \{ wtw^{-1}\} ,w}) =
v_{sJs,swt} +(u_t-1) v_{sJs \cup \{ swtw^{-1}s \},swt}  + (u_t-1)v_{sJs \cup \{ swtw^{-1}s\} ,sw}
$
while
$G'_t G_s v_{J,w} = G'_t v_{sJs,sw} = 
v_{sJs,swt} +(u_t-1) v_{sJs \cup \{ swtw^{-1} \}s,swt}  + (u_t-1)v_{sJs \cup \{ swtw^{-1}s\} ,sw}
$ hence $G'_t G_s v_{J,w} = G_s G'_t v_{J,w} $.

If $\ell(sw) = \ell(w)-1$ and $\ell(wt) = \ell(w)+1$, then we have $\ell(swt) = \ell(w)$
for the same reason as in the preceding case. Then
$G_s G'_t v_{J,w} = G_s v_{J,wt} = 
v_{sJs,swt} + (u_s-1) v_{sJs\cup \{ s \},swt} + (u_s-1) v_{sJs \cup \{ s \},wt}$
while 
$G'_t G_s v_{J,w} =
G'_t(v_{sJs,sw} + (u_s-1) v_{sJs\cup \{ s \},sw} + (u_s-1) v_{sJs \cup \{ s \},w})
=
v_{sJs,swt} 
+ (u_s-1) v_{sJs\cup \{ s \},swt} 
+ (u_s-1) v_{sJs \cup \{ s \},wt}
$
hence $G'_t G_s v_{J,w} = G_s G'_t v_{J,w} $.

If $\ell(sw) = \ell(w)-1 = \ell(wt)$, then
\begin{itemize}
\item either $\ell(swt) = \ell(wt)-1 = \ell(sw)-1$, in which case
$$
\begin{array}{lcl}
G_s G'_t v_{J,w} 
&=& G_s (
v_{J,wt} +(u_t-1) v_{J \cup \{ wtw^{-1} \},wt}  + (u_t-1)v_{J \cup \{ wtw^{-1}\} ,w}) \\
&=& 
 v_{sJs,swt} + (u_s-1) v_{sJs\cup \{ s \},swt} + (u_s-1) v_{sJs \cup \{ s \},wt} 
+(u_t-1)(v_{sJs \cup \{ swtw^{-1} s\} ,swt}\\ &+ &  (u_s-1) v_{sJs\cup \{ s \} \cup \{ swtw^{-1}s \},swt} + (u_s-1) v_{sJs \cup \{ s \} \cup \{ swtw^{-1}s \},wt})
\\&+ & (u_t-1)(
v_{sJs \cup \{ swtw^{-1}s \},sw} + (u_s-1) v_{sJs\cup \{ s \}\cup \{ swtw^{-1}s \},sw} \\ & +& (u_s-1) v_{sJs \cup \{ s \}\cup \{ swtw^{-1}s \},w})
\end{array}
$$
 and
$$
\begin{array}{lcl}
G'_tG_s v_{J,w} &=& G'_t(v_{sJs,sw} + (u_s-1) v_{sJs\cup \{ s \},sw} + (u_s-1) v_{sJs \cup \{ s \},w} ) \\
&=&
v_{sJs,swt} +(u_t-1) v_{sJs \cup \{ swtw^{-1}s \},swt}  + (u_t-1)v_{sJs \cup \{ swtw^{-1}s\} ,sw}
\\ & +&  (u_s-1) (
v_{sJs \cup \{ s \},swt} 
+(u_t-1) v_{sJs \cup \{ s \} \cup \{ swtw^{-1}s \},swt}  + (u_t-1)v_{sJs \cup \{ s \} \cup \{ swtw^{-1}s\} ,sw})
\\ &+ &  (u_s-1) (
v_{sJs \cup \{ s \},wt} +(u_t-1) v_{sJs \cup \{ s \} \cup \{ wtw^{-1} \},wt}  + (u_t-1)v_{sJs \cup \{ s\} \cup \{ wtw^{-1}\} ,w})
\end{array}
$$
Therefore, these terms are equal as soon as we have
$$
 v_{sJs \cup \{ s \} \cup \{ swtw^{-1}s \},wt}
 +  v_{sJs \cup \{ s \}\cup \{ swtw^{-1}s \},w}
 = 
  v_{sJs \cup \{ s \} \cup \{ wtw^{-1} \},wt}  + v_{sJs \cup \{ s\} \cup \{ wtw^{-1}\} ,w}.
  $$
Since $$\overline{sJs \cup \{ s \} \cup \{ swtw^{-1}s \}} = \overline{sJs \cup \{ s \} \cup \{ wtw^{-1} \}}$$
this holds true.
\item or $\ell(swt) = \ell(w)$. But since $\ell(sw) = \ell(wt)$ this implies $sw = wt$.
Then 
$$
\begin{array}{lcl}
G_s G'_t v_{J,w} &=& G_s(v_{J,wt} +(u_t-1) v_{J \cup \{ wtw^{-1} \},wt}  + (u_t-1)v_{J \cup \{ wtw^{-1}\} ,w}) \\ &=& 
v_{sJs,swt} +(u_t-1) v_{sJs \cup \{ swtw^{-1}s^{-1} \},swt}  + (u_t-1)(
v_{sJs\cup \{ swtw^{-1}s^{-1}\} ,sw}\\ & & + (u_s-1) v_{sJs\cup \{ s \}\cup \{ swtw^{-1}s^{-1}\} ,sw} + (u_s-1) v_{sJs \cup \{ s \}\cup \{ swtw^{-1}s^{-1}\} ,w}
)
\end{array}$$
and 
$$
\begin{array}{lcl}
G'_t G_s v_{J,w} &=& G'_t (v_{sJs,sw} + (u_s-1) v_{sJs\cup \{ s \},sw} + (u_s-1) v_{sJs \cup \{ s \},w} ) \\
&=& 
v_{sJs,swt} 
+ (u_s-1) v_{sJs\cup \{ s \},swt} 
+ (u_s-1) G'_t v_{sJs \cup \{ s \},w}  \\
&=&
v_{sJs,swt} 
+ (u_s-1) v_{sJs\cup \{ s \},swt} 
+ (u_s-1)(
v_{sJs \cup \{ s \},wt}\\ & & +(u_s-1) v_{sJs \cup \{ s \} \cup \{ wtw^{-1} \},wt}  + (u_t-1)v_{sJs \cup \{ s \} \cup \{ wtw^{-1}\} ,w}
)
\end{array}
$$
 Since $sw = wt$ implies $swt = w$, $s = wtw^{-1}$, $swtw^{-1} s^{-1} = s$
and $u_s = u_t$, these two expressions are equal.
\end{itemize}
We thus proved that the $G_s,E_s$ commute with the $G'_t$, $E'_t$ for $s,t \in S$.

We finally define, for $K \subset \mathcal{R}$, the endomorphism $E_K \in \End(V)$
by $E_K.v_{J,w} = v_{J \cup K,w}$. For $s \in S$ we have
$G_s E_K. v_{J,w} = G_s. v_{J \cup K,w} = v_{sJs \cup sKs,sw}$
and $E_{sKs} G_s. v_{J,w} =E_{sKs} v_{sJs,sw} = v_{sKs \cup sJs,sw}$ if $\ell(sw) = \ell(w)+1$,
while 
$G_s E_K. v_{J,w} = G_s. v_{J \cup K,w} =
v_{sJs \cup sKs,sw} + (u_s-1)v_{sJs\cup sKs \cup \{ s \},sw} + (u_s-1) v_{sJs \cup sKs \cup
\{ s \},w}$
and
$E_{sKs} G_s v_{J,w} = E_{sKs}(v_{sJs,w}+(u_s-1) v_{sJs\cup \{ s \},sw} + (u_s-1)
v_{sJs \cup \{s\},w}) =
v_{sJs \cup sKs,sw} + (u_s-1)v_{sJs\cup sKs \cup \{ s \},sw} + (u_s-1) v_{sJs \cup sKs \cup
\{ s \},w}$
 it $\ell(sw) = \ell(w)-1$. This proves $G_s E_K = E_{s(K)} G_s$
 for all $s \in S$.

Now, for $s,t \in S$, we denote $m_{st}$ the order of $st$ in $W$. We let
$$\omega 
= \underbrace{sts\dots}_{m_{st}} 
= \underbrace{tst\dots}_{m_{st}}  \in W.
$$
Then,
$$
\underbrace{G_sG_tG_s\dots}_{m_{st}} v_{\emptyset,1} = 
v_{\emptyset,\underbrace{sts\dots}_{m_{st}}} = 
v_{\emptyset,\underbrace{tst\dots}_{m_{st}}} = 
\underbrace{G_tG_sG_t\dots}_{m_{st}} v_{\emptyset,1} 
$$
hence, 
writing $w$ as $t_1\dots t_r$ with $t_i \in S$, 
we
have
$$
\begin{array}{lcl}
\underbrace{G_sG_tG_s\dots}_{m_{st}} v_{J,w} &=& 
\underbrace{G_sG_tG_s\dots}_{m_{st}} G'_{t_r}\dots G'_{t_2}G'_{t_1}v_{J,1} \\
&=& 
G'_{t_r}\dots G'_{t_2}G'_{t_1} \underbrace{G_sG_tG_s\dots}_{m_{st}} v_{J,1} \\
&=& 
G'_{t_r}\dots G'_{t_2}G'_{t_1} \underbrace{G_sG_tG_s\dots}_{m_{st}} E_J v_{J,1} \\
&=& 
G'_{t_r}\dots G'_{t_2}G'_{t_1} E_{\omega J\omega^{-1}} \underbrace{G_sG_tG_s\dots}_{m_{st}}  v_{\emptyset,1} \\
&=& 
E_{\omega J\omega^{-1}} G'_{t_r}\dots G'_{t_2}G'_{t_1} \underbrace{G_sG_tG_s\dots}_{m_{st}}  v_{\emptyset,1} \\
&=& 
E_{\omega J\omega^{-1}} G'_{t_r}\dots G'_{t_2}G'_{t_1} \underbrace{G_tG_sG_t\dots}_{m_{st}}  v_{\emptyset,1} \\
&=& 
E_{\omega J\omega^{-1}}  \underbrace{G_tG_sG_t\dots}_{m_{st}} G'_{t_r}\dots G'_{t_2}G'_{t_1} v_{\emptyset,1} \\
&=& 
E_{\omega J\omega^{-1}}  \underbrace{G_tG_sG_t\dots}_{m_{st}}  v_{\emptyset,w} \\
&=& 
\underbrace{G_tG_sG_t\dots}_{m_{st}} E_{J}    v_{\emptyset,w} \\
&=& 
\underbrace{G_tG_sG_t\dots}_{m_{st}} E_{J}    v_{J,w} 
\end{array}
 $$
From this we get that the map $g_s \mapsto G_s$, $e_J \mapsto E_J$
induces a $\kk$-algebra homomorphism $\mathcal{C}_W(\underline{u}) \to \End(V)$.
We let $A$ denote its image. Since the $e_{\overline{J}}g_w$ span $\mathcal{C}_W(\underline{u})$
and their image maps $v_{\emptyset,1}$ to $v_{\overline{J},w}$ we get that
this homomorphism is injective, and that its image surjects onto the free $\kk$-module $V$
under the map $a \mapsto a. v_{\emptyset,1}$. This proves the claim.
\end{proof}

\subsection{An extension of the Iwahori-Hecke algebra}

The algebra $\mathcal{C}_W(\underline{u})$ is an extension of the Iwahori-Hecke algebra $H_W(\underline{u})$.
We let $T_s,s \in S$ denote the natural generators of $H_W(\underline{u})$, and $T_w = T_{s_1}\dots T_{s_r}$
when $w = s_1\dots s_r$ is a reduced expression of $w \in W$.

\begin{proposition} \label{prop:exthecke} Let $(W,S)$ denote a Coxeter system.
\begin{enumerate}
\item The map $g_s \mapsto T_s$, $e_s \mapsto 1$ induces a surjective $\kk$-algebra morphism $\mathfrak{p} : \mathcal{C}_W(\underline{u}) \onto H_W(\underline{u})$.
For $w \in W$, it maps $g_w$ to $T_w$ and each $e_J$ to $1$. Its kernel is the two-sided ideal generated by the $e_s-1$, $s \in S$.
\item If $S$ is finite, then $\mathfrak{p}$ is split. A splitting is given by $T_w \mapsto g_w e_S$, with $e_S = e_W = \prod_{s \in S} e_s$.
\end{enumerate}
\end{proposition}
\begin{proof}
One gets that the map $g_s \mapsto T_s$, $e_s \mapsto 1$ induces a morphism of (unital) $\kk$-algebras $\mathfrak{p} : \mathcal{C}_W(\underline{u}) \to H_W(\underline{u})$,
by checking that the defining relations of $\mathcal{C}_W(\underline{u})$ hold inside $H_W(\underline{u})$. This is immediate for relations (1)-(5),
and (6) is mapped to the defining relation $T_s^2 = u_s+(u_s-1)T_s$ of $H_W(\underline{u})$. This morphism is surjective because the $T_s$'s generate $H_W(\underline{u})$
as a unital $\kk$-algebra. By definition of $g_w$ and $T_w$ it is clear that $\mathfrak{p}(g_w) = T_w$ for all $w \in W$,
and similarly that $\mathfrak{p}(e_J) = 1$ for all $J$'s.
By theorem \ref{thm:main} we know that $\mathcal{C}_W(\underline{u})$ is spanned by the $g_w e_J$, with $w \in W$ and $J \in \mathcal{P}_f(\mathcal{R})$.
An element $x \in \Ker \mathfrak{p}$ can be written $\sum_{w,J} a_{w,J} g_w e_J$ with $a_{w,J} \in \kk$ almost all zero,
such that $0 = \sum_{w,J} a_{w,J} T_w = \sum_w \left(\sum_J a_{w,J}\right)T_w$. Let us fix $w \in W$, and let $b_J =a_{w,J}$. We have
$\sum b_J =0$ since the $T_w$'s form a basis of $H_W(\underline{u})$, so it is sufficient to prove that every element in $x \in \mathfrak{p}$
of the form $\sum_J b_J e_J$  belongs to the ideal $\mathfrak{I}$ generated by the $e_s-1, s \in S$. 
This amounts to saying that $e_J -1 \in \mathfrak{I}$ for all $J$. Letting $r(W_0)$ denotes the minimal number of reflections needed for generating $W_0$,
we prove this by induction on $r(\langle J \rangle)$. The case $r(\langle J \rangle)=0$ is obvious, the case $r(\langle J \rangle) = 1$
is a consequence of $g_w (e_s-1)g_w^{-1} = e_{wsw^{-1}}-1$ for all $w \in W$ and $s \in S$. Now, if $r(\langle J \rangle)> 1$, there exists
$t \in J$ such that $r(\langle K \rangle) < r(\langle J \rangle)$, where $K = J \setminus \{ t \}$. Again because 
$g_w (e_J-1)g_w^{-1} = e_{wJw^{-1}}-1$, we can assume $s \in S$. Then, $e_J = e_K e_s$ and $e_J-1 = e_K(e_s-1)+e_K-1 \in e_K-1 + \mathfrak{I}$,
so we get $e_J-1 \in \mathfrak{I}$ by the induction assumption. This completes the proof of (1).

In order to prove (2), we first note that $e_W$ is central and idempotent.
We prove that $T_s \mapsto g_s e_W$, $1 \mapsto e_W$ induces an algebra morphism. Since $e_W$ is central, the braid relations $T_s T_t T_s \dots = T_tT_sT_t \dots$
are mapped to $e_W^{m_{st}}g_s g_t g_s \dots = e_W^{m_{st}}g_t g_s g_t \dots$
and this holds true inside $\mathcal{C}_W(\underline{u})$. The quadratic relation $T_s^2 = (u_s-1)T_s + u_s$
is mapped to $g_s^2 e_W = (u_s-1)g_s e_W + u_se_W$. We prove that this holds true, because the
relation $g_s^2 = 1 + (u_s -1)e_s(1+g_s)$ implies
$g_s^2 e_s = e_s + (u_s-1)e_s (1+g_s) = u_s e_s + (u_s-1)g_s$ and therefore,
since $e_se_W = e_W$, we get $g_s^2 e_W = u_s e_W + (u_s-1)g_s e_W$. Therefore
there exists a $\kk$-algebra morphism $\mathfrak{q} : H_W(\underline{u}) \to \mathcal{C}_W(\underline{u})$,
which maps $T_w$ to $g_w e_W$ as is readily checked by induction on $\ell(w)$. We have
$\mathfrak{p}(\mathfrak{q}(T_w)) = \mathfrak{p}(g_w e_W) = T_w$, and this proves (2).

\end{proof}

\subsection{Meaningful quotients}
\label{sect:meaning}

We recall that $\mathcal{W}$ denotes the collection of finitely generated reflection subgroups of $W$,
endowed with the conjugation action of $W$. If $J \in \mathcal{P}_f(\mathcal{R})$,
we let $e_{\langle J \rangle} = e_{\overline{J}} = e_J$.
The algebra $\mathcal{C}_W(\underline{u})$ is spanned by elements $e_{\langle J \rangle} g_w$
for $w \in W$ and $\langle J \rangle \in \mathcal{W}$. Let $\mathcal{F}$ be a $W$-set
and $p = \mathcal{W} \to \mathcal{F}$
be a surjective map which is $W$-equivariant. Such a map can be seen
as an equivalence relation on $\mathcal{W}$ compatible with the
action of $W$. We also assume that $p(\langle J\rangle) = p(\langle K\rangle)$
implies $p(\langle J,s\rangle) = p(\langle K,s\rangle)$ for all $s \in S$.

\begin{proposition} Let $p : \mathcal{W} \onto \mathcal{F}$
be as above,
and $I_p$ the ideal of $\mathcal{C}_W(\underline{u})$
generated by the $e_{J} - e_{K}$ for $p(J) = p(K)$.
The quotient algebra $\mathcal{C}_W^{\mathcal{F}}(\underline{u}) = \mathcal{C}_W(\underline{u})/I_p$ is a free module,
of rank $|W|.|\mathcal{F}|$ if $W$ is finite. The algebra morphism $\mathfrak{p} : \mathcal{C}_W(\underline{u}) \onto H_W(\underline{u})$
factorizes through the natural projection $\mathcal{C}_W(\underline{u}) \onto \mathcal{C}_W^{\mathcal{F}}(\underline{u})$.
\end{proposition}
\begin{proof}
Let $I'_p$ be the $\kk$-module spanned by the 
$(e_{\langle J\rangle} - e_{\langle K\rangle})g_w$ for $w \in W$ and
$p(\langle J\rangle) = p(\langle K\rangle)$.
Since $p(\langle J\rangle) = p(\langle K\rangle)$
implies $p(\langle J,s\rangle) = p(\langle K,s\rangle)$ we know that
$e_sI'_p \subset I'_p$ for all $s \in S$;
since $p$ is equivariant we have $g_w I'_p g_w^{-1} \subset I'_p$ for all $w \in W$
and therefore $I'_p e_s \subset I'_p$. From this and the defining relation (6) of $\mathcal{C}_W(\underline{u})$
we get $I'_p g_s \subset I'_p$ for all $s \in S$, and $g_s I'_p = g_s I'_p g_s^{-1}.g_s \subset I'_p$.
Therefore $I'_p$ is an ideal. Since $I'_p \subset I_p$ we get $I_p = I'_p$ hence $I_p$
is spanned by the 
$(e_{\langle J\rangle} - e_{\langle K\rangle})g_w$ for $w \in W$ and
$p(\langle J\rangle) = p(\langle K\rangle)$. The assertion on the structure as a module and the rank then follows from the previous theorem.
The factorization assertion is clear from the definition of $I_p$ and proposition \ref{prop:exthecke}.
\end{proof}

Important examples of such $p$ are the following ones :
\begin{enumerate}
\item $\mathcal{F}=\mathcal{F}_{parab}$ is the collection of parabolic subgroups, and the map $p$
associates to $G \in \mathcal{W}$ the fixer of the fixed-point set $\{ x \in \R^n ; \forall g \in G
\ \ g. x =x \}$
\item If $W$ is the Weyl group of a reduced root system $R$, then
$\mathcal{W}$ can be identified with the collection of root subsystems of $R$. Then, 
one can take for $\mathcal{F}=\mathcal{F}_{closed}(R)$ the collection of closed symmetric subsystems,
and for $p$ the map which associate to a root subsystem its closure.
\end{enumerate}

The first example arises for arbitrary groups, and is the smaller of the two types,
when both can be compared : there is a natural surjective map $\mathcal{F}_{closed}(R) \to
\mathcal{F}_{parab}$ which is not bijective in general (e.g. see $A_2$ as the set of long roots inside $G_2$).
The second one is the one which is the
most relevant to the original Yokonuma-Hecke algebra $\mathcal{H}(G,U)$, as
$\mathcal{C}_W^{\mathcal{F}_{closed}(R)}$ `generically' embeds into $\mathcal{H}(G,U)$
(see section \ref{sect:YHvintage}). For short, we let $\mathcal{C}_W^R(\underline{u}) =  \mathcal{C}_W^{\mathcal{F}_{closed}(R)}(\underline{u})$
and $\mathcal{C}_W^p(\underline{u}) =  \mathcal{C}_W^{\mathcal{F}_{parab}}(\underline{u})$.

Note that, when $W$ has type $A_n$, and $R$ is the root system of type $A_n$, then
$\mathcal{C}_W(\underline{u})
= \mathcal{C}_W^{R}(\underline{u})=\mathcal{C}^{p}_W(\underline{u})$.
Moreover, in general the morphism onto $H_W(\underline{u})$ factorizes as follows
$$
\mathcal{C}_W(\underline{u}) \onto \mathcal{C}^{R}_W(\underline{u}) \onto \mathcal{C}^p_W(\underline{u}) \onto H_W(\underline{u}).
$$

\begin{comment}
From the investigation of section ??? one readily gets the following

\begin{proposition}
Let $G$ be a Chevalley group over $\F_q$ with Weyl group $W$ and root system $R$.
The natural map $\mathcal{C}_W(q) \to \mathcal{H}(G,U)$ factorizes through the natural projection
$\mathcal{C}_W(q) \onto \mathcal{C}_W^R(q)$, and the induced map $\mathcal{C}_W^R(q) \to \mathcal{H}(G,U)$
is injective when \dots
\end{proposition}
\begin{proof}
By proposition \ref{prop:relsidsHGU} (4)

\end{proof}
\end{comment}

\subsection{Lusztig's involution and Kazhdan-Lusztig bases}

Our basic reference on Kazhdan-Lusztig bases is \cite{SOERGEL}, although it deals only with the 1-parameter case, but the properties
that we use here are easily generalized from this case. The general statements can also be found in \cite{BONNAFE} (see also \cite{LUSZTIGBOOK} for an intermediate
situation).
%For the aspects for which  general case for the 1-parameter case, and \cite{LUSZTIGBOOK} for the general case.
We choose $\kk = \Z[v_s,v_s^{-1} ; s \in S]$, where there is one formal parameter
$v_s$ for each $s \in S$, with the condition $v_s = v_t$ if $s \sim t$. For short, we denote it $\kk = \Z[v,v^{-1}]$. We set $u_s = v_s^2$ for each $s \in S$.
%set $u_s = v_s^2$, that is we define $\mathcal{C}_W$
%over the ring $R_0 = \Z[v_s,v_s^{-1}]$ extending $R = \dots$.
We set $H_s = (-v_s^{-1})g_s$ for $s \in S$ and 
$H_w = H_{s_1} \dots H_{s_m}$ for $w \in W$ and $w = s_1\dots s_m$ a minimal
decomposition. Let $H_w^0$ denote its image inside the Hecke algebra $H_W$.
We have $(H_s^0)^2 = 1 + (v_s^{-1} - v_s)H_s^0$, and Lusztig's involution is
defined on $H_W$ by $H_w^0 \mapsto (H_{w^{-1}}^0)^{-1}$, $v_s \mapsto v_s^{-1}$.
%more generally $H_w = (-v^{-1})^{\ell(w)} g_w$ for $w \in W$.
The following equalities are easily checked
\begin{comment}
$$
\begin{array}{lcl}
H_s^{-1} & =& v^2 H_s + (v-v^{-1}) e_s (1- v H_s) \\
H_s^2 &=& v^{-2} + (v-v^{-1}) e_s (v^{-1} - H_s) \\
H_s^{-2} &=& v^2 + (v^{-1} - v)e_s (v- H_s^{-1})
\end{array}
$$
\end{comment}
{}
$$
\begin{array}{lcl}
H_s^{-1} & =& v_s^2 H_s + (v_s-v_s^{-1}) e_s (1- v_s H_s) \\
H_s^2 &=& v_s^{-2} + (v_s-v_s^{-1}) e_s (v_s^{-1} - H_s) \\
H_s^{-2} &=& v_s^2 + (v_s^{-1} - v_s)e_s (v_s- H_s^{-1})
\end{array}
$$
Moreover, we have $H_s e_{W_0} H_s^{-1} = e_{s W_0 s^{-1}} = H_s^{-1} e_{W_0} H_s$.
From this the following proposition readily follows.

\begin{proposition} There exists an involutive ring automorphism of $\mathcal{C}_W$ over $\kk = \Z[v,v^{-1}]$ which maps
$v_s \mapsto v_s^{-1}$, $H_w \mapsto H_{w^{-1}}^{-1}$, $e_{W_0} \mapsto e_{W_0}$ for each $w \in W$ and $W_0 \in \mathcal{W}$.
It induces similar automorphisms of $\mathcal{C}_W^p$ and $\mathcal{C}_W^R$ (when defined).
 It is compatible with the ring automorphism of $\Z[v,v^{-1}] B$
which maps $v_s \mapsto v_s^{-1}$, $s \mapsto s^{-1}$ for $s \in S$, and with Lusztig's involution of $H_W$ (as in \cite{SOERGEL,LUSZTIGBOOK}),
that is the following diagram commutes, where the vertical maps are these involutive automorphisms and the horizontal
ones are the natural maps.
%It induces an involutive automorphism of $H_W$ that coincides with Lusztig's involution,
%that is $H_w
%as in \cite{SOERGEL}, where $H_w$ is mapped to the element denoted $H_w$ in \cite{
$$\xymatrix{
 \Z[v,v^{-1}] B \ar[r] \ar[d] & \mathcal{C}_W \ar[r] \ar[d] & H_W  \ar[d] \\
  \Z[v,v^{-1}] B \ar[r] & \mathcal{C}_W \ar[r]  & H_W  \\
}
$$
\end{proposition}

Assume that all reflections of $W$ are conjugate. Recall from \cite{SOERGEL} that the Kazhdan-Lusztig basis $(\mathbf{H}^0_w)_{w \in W}$ of $H_W$ is characterised
by the properties $\overline{ \mathbf{H}^0}_w = \mathbf{H}^0_w$  and
$\mathbf{H}^0_w \in H^0_w + \sum_{y \in w} v \Z[v] H^0_y$. One readily checks that,
for $\mathbf{H}_s \in C_W$ and $W = A_1$, the conditions $\overline{\mathbf{H}_s} = \mathbf{H}_s$
and $\mathbf{H}_s \in H_s + \sum_{y \in W, W_0 < W} v \Z[v] H_y e_{W_0}$
are equivalent to saying that $\mathbf{H}_s = H_s + (xv + v^2)(1-e_s) H_s + v e_s$ for
\emph{some} $x \in \Z$. Note that such a $\mathbf{H}_s$ clearly maps onto $\mathbf{H}^0_s = H_s + v$.
The question of whether the Kazhdan-Lusztig basis can be `lifted' in a natural and essentially unique way
is therefore an intriguing one, that we leave open for now.

\subsection{Combinatorics and Bell numbers}
\label{sect:Bell}

In type $A_{n-1}$, reflections have the form $(i,j), 1 \leq i<j \leq n$, and therefore a subset of $\mathcal{R}$
can be identified with a graph on $n$ vertices. If $J \subset \mathcal{R}$, then $\bar{J}$ is the graph
of the transitive closure of the graph given by $J$, and the set of all graphs of this form
is the set of disjoint unions of complete graphs on $\{1,\dots, n \}$. This set is in natural 1-1 correspondence
with partitions of the set $\{1,\dots, n\}$, and therefore has for cardinality the
$n$-th Bell number $Bell_n$ : $1,1,2,5,15,52,203,877,\dots$. Because of this, we will call
in general the Bell number of type $W$ the number of reflection subgroups of $W$,
and we will call $W$-partitions the elements $\bar{J}$, $J \subset \mathcal{R}$.

In type $D_n$, it can be interpreted as the number of symmetric partitions
of $\{-n,\dots,n \} \setminus \{ 0 \}$ such that none of the subsets is of the form
$\{j,-j \}$, see sequence A086365 in Sloane's Online Encyclopaedia of Integer Sequences. Here symmetric means that, for
every part $X$ of the partition, its opposite $-X$ is a part of the partition.

Indeed, the reflections have the form $s_{ij}$ or $s_{ij}'$,
where 
$$
\begin{array}{lcl}
s_{ij}.(z_1,\dots,z_i,\dots,z_j ,\dots, z_n)
&=& (z_1,\dots,z_j,\dots,z_i ,\dots, z_n)\\
s_{ij}'.(z_1,\dots,z_i,\dots,z_j ,\dots, z_n) &=& (z_1,\dots,-z_j,\dots,z_i ,\dots, z_n);
\end{array}
$$
then, to a stable subset $\mathcal{R}_0$ of $\mathcal{R}$
we associate the partition
of $\{-n,\dots,n \} \setminus \{ 0 \}$ made of the equivalence classes
under the relation $i \sim j$ for $ij > 0$ if $s_{ij} \in \mathcal{R}_0$,
for $ij <0$ if $s_{ij}' \in \mathcal{R}_0$. Conversely, we associate to a partition
$\mathcal{P}$ the collection of reflections made of the $s_{ij}$ for $i,j > 0$
in the same part of $\mathcal{P}$, and of the $s_{ij}'$ for $i,j > 0$ when $-i,j$ belong
to the same part of $\mathcal{P}$. These two maps provide a bijective correspondence.
An exponential generating function for this sequence is
$$
-1 + \exp\left( -x + e^{x}-1 + \frac{e^{2x}-1}{2} \right)
$$
and
the first terms are $1,4,15,75,428$. 
In type $B_n, n \geq 2$, we get the numbers
$8,38,218,1430,10514,\dots$, which we could not relate to other
mathematical objects. In type $I_2(m)$, we get $1+ \sigma(m)$,
where $\sigma(m)$ is the sum of divisors of $m$. Indeed, 
the non-trivial reflection subgroups are the stabilizer of the $d$-gons
with vertices $\exp(2 \pi \ii (\frac{k_0}{m}+\frac{k}{d}))$
for some $k_0 \in [0,m/d[$ and $k$ running from $0$ to $d-1$, and $d$
is a divisor of $m$. Since there are $m/d$ such $d$-gons for
$d$ dividing $m$, there are exactly $\sigma(m)$ non-trivial reflection subgroups.

\begin{table}
$$
\begin{array}{|c|c|c|c|c|r|}
\hline
W & |W| & \mathrm{Bell}^p(W) & \mathrm{Bell}^R(W)&\mathrm{Bell}(W) & \mathrm{rk}\, \mathcal{C}_W(\underline{u}) \\
\hline
G_2 & 12 & 8 & 12 & 13 & 156 \\
\hline
H_3 &120 &48 & & 53 & 6360 \\
\hline
H_4 & 14400& 2104& &2760 & 39744000 \\
\hline
F_4 & 1152& 268&447 & 637 & 733824 \\
\hline
E_6 & 51840& 4598&  5079 & 5079 &  263295360 \\
\hline
E_7 & 2903040 & 90408&  107911 &107911 &  313269949440 \\
\hline
E_8 &  696729600& 5506504 &  7591975 &7591975 & 5289553704960000 \\
\hline
\end{array}
$$
\caption{Bell numbers in exceptional types.}
\label{table:bellexc}
\end{table}

Among the exceptional groups, we computed the number of reflection subgroups
by using elementary methods in the computer system GAP3 together with its CHEVIE
package, except for the
largest ones $E_7$ and $E_8$, for which this was not sufficient. Therefore, we used
the classification of their reflection subgroups provided in \cite{DPR} in this case : the total
number is then the sum of the number of conjugacy classes provided in the third
columns of tables 4 and 5 of \cite{DPR}. The result can be found in table \ref{table:bellexc}.
In order to find the dimension of $\mathcal{C}^p(W)$, we need to know the number of parabolic subgroups. These are in 1-1 correspondence with the elements of the lattice of the corresponding
hyperplane arrangements, and with this interpretation they are described in \cite{OT}. We call
parabolic Bell number of type $W$ and denote $Bell^p(W)$ this number. 
Finally, when $R$ is (one of) the classical root systems
attached to $W$, we call Bell number of type $R$ and denote $Bell^{R}(W)$ the number of
closed root subsystems. If $W$ is of simply laced (ADE) type, then $Bell^{R}(W) = Bell(W)$.
For exceptional groups, both numbers are also
listed in table \ref{table:bellexc}. 
\begin{table}
$$
\begin{array}{|c||c|c|c|c|c|c|}
\hline
n & 2 & 3 & 4 & 5 & 6 & 7 \\ % 7 & 8 \\
\hline
\hline
Bell(B_n) & 8 & 38 & 218 & 1430 & 10514 & 85202 \\
\hline
Bell^R(B_n) & 7 & 31 & 164 & 999  & 6841 & 51790 \\
\hline
Bell^p(B_n) & 6 & 24 &116  & 648 &4088  & 28640  \\
\hline
\hline
Bell^{(R)}(D_n) & 4 & 15 & 75 & 428 & 2781 & 20093 \\ %& 20093 \\
%\hline
%Bell^R(D_n) & \\
\hline
Bell^p(D_n) & 4 & 15 & 72  & 403  & 2546  & 17867   \\
\hline
\end{array}
$$
\caption{Bell numbers in classical types.}
\label{table:bellcla}
\end{table}

For the infinite series $B_n$ and $D_n$, the first values are
listed in table \ref{table:bellcla}. The series $Bell^p(D_n)$ and $Bell^p(B_n)$ are investigated and presented as analogues
of Bell numbers in \cite{SUTER}. J. East communicated to us that he too generalized Bell
numbers to series $B,D$ and $I_2(m)$ (unpublished). In his approach, the `right analogues'
are $Bell^R(B_n)$, $Bell(D_n)$ and $Bell(I_2(m))$, respectively, which correspond
to the sequences A002872, A086365 and A088580 in Sloane's encyclopaedia of integer sequences.
To the best of our knowledge, the sequence $Bell(B_n)$ has not yet been investigated.

\subsection{Specialization at $u=1$ and semisimplicity}
\label{sect:spectss}

The algebra $\mathcal{C}_W(1)$ is obviously a semidirect product $\kk W \ltimes A$,
where $A$ is the subalgebra generated by the idempotents $e_{J}$.

Let $L$ be a join semilattice. That is, we have a finite partially ordered set $L$
for which there exists a least upper bound $x \vee y$ for every two $x,y \in L$.
Let $M$ be the semigroup with elements $e_{\la}, \la \in L$ and product law
$e_{\la} e_{\mu} = e_{\la \vee \mu}$. 
Such a semigroup is sometimes called a band.

If $L$ is acted upon by some group $G$
in an order-preserving way (that is $x \leq y \Rightarrow g.x \leq g.y$ for all $x,y \in L$ and $g \in G$)
then $M$ is acted upon by $G$, so that we can form the algebra $ \kk M \rtimes \kk G$. Up
to exchanging meet and joint, the algebra $\kk M$ is the Möbius algebra of \cite{STANLEY}, definition 3.9.1
(this reference was communicated to us by V. Reiner). We will need the following proposition, which is in part a $G$-equivariant version of \cite{STANLEY}, theorem 3.9.2.
Here $\kk^{L}$ is the algebra of $\kk$-valued functions on $L$, that is the direct product of a collection indexed by the elements of $L$ of copies of the $\kk$-algebra $\kk$.

\begin{proposition}  Let $M$ be the band associated to a 
finite join semilattice $L$. For every commutative ring $\kk$, the
semigroup algebra $\kk M$ is isomorphic to $\kk^{L}$.
If $L$ is acted upon
by some group $G$ as above, then $\kk M \rtimes \kk G \simeq \kk^{L} \rtimes \kk G$.
If $G$ is finite and
$\kk$ is a field whose  characteristic 
does not divide $|G|$, then the algebra $ \kk M \rtimes \kk G$ is semisimple. If $\kk G_{\la}$ is split semisimple for all $\la \in L$, where $G_{\la}<G$ is the stabilizer of $\la$, then so is $ \kk M \rtimes \kk G$.
\end{proposition}
\begin{proof}
To each $\la \in L$ we associate $\varphi_{\la} : L \to \kk$ defined by $\varphi_{\la}(\mu) = 1$ if $\la \leq \mu$
and $\varphi_{\la}(\mu) =0$ otherwise. We define a $\kk$-linear map $c : M \to \kk^{L}$
by $e_{\la} \mapsto \varphi_{\la}$. We prove that $c$ is an algebra homomorphism. We have
that $\varphi_{\la_1} \varphi_{\la_2}$ maps $\mu \in L$ to $1$ iff $\la_1 \leq \mu$ and $\la_2 \leq \mu$,
and to $0$ otherwise ; $\varphi_{\la_1 \vee \la_2}$ maps $\mu \in L$ to $1$ iff $\la_1 \vee \la_2 \leq \mu$,
and to $0$ otherwise. These two conditions being equivalent, this proves $c(e_{\la_1}e_{\la_2})=
c(e_{\la_1})c(e_{\la_2})$, hence $c$ is a $\kk$-algebra homomorphism. We now prove that $c$
is injective. We assume $\sum_{\la \in L} a_{\la} \varphi_{\la} = 0$ for a collection of $a_{\la} \in \kk$,
and we want to prove that all $a_{\la}$'s are zero. If not, let $\la_0$ be a minimal element (w.r.t. $\leq$)
among the elements of $L$ such that $a_{\la} \neq 0$. Then
$0=\sum_{\la \in L} a_{\la} \varphi_{\la}(\la_0) = a_{\la_0}$ provides a contradiction. Therefore, $c$ is injective. We now prove that $c$ is surjective. Let $f_{\la} \in \kk^{L}$
being defined by $f_{\la}(\mu) = \delta_{\la,\mu}$ (Kronecker symbol). The $f_{\la}$'s obviously form
a basis of $\kk^{L}$ and we need to prove that they belong to the image of $c$, that is to the submodule
$V$ spanned
by the $\varphi_{\la}$'s.
Let $\la_0 \in L$. We prove that $f_{\la_0}$ belongs to $V$ by induction with respect to $\leq$. If $\la_0$ is minimal in $L$, then $\varphi_{\la_0} = f_{\la_0}$ and this holds true. Now assume $f_{\la} \in V$
for all $\la < \la_0$. Let $g = f_{\la_0} - \varphi_{\la_0}$. We have $g(\mu) = 0$ unless $\mu < \la_0$.
Therefore $g$ is a linear combination of the $f_{\mu}$'s for $\mu < \la_0$ hence $g \in V$
and this implies $f_{\la_0} \in \varphi_{\la_0} + V \subset V$. By induction we conclude
that $c$ is surjective, and therefore is an isomorphism.

Now assume that 
$L$ is acted upon by $G$. Then $\kk M$ and $\kk^{L}$ are
both natural $\kk G$-modules : if $g \in G$, then $g.e_{\la} = e_{g. \la}$ and,
if $f : L \to \kk$, then $g.f : \la \mapsto f(g^{-1}.\la)$. For these actions, $c$ is an isomorphism
of $\kk G$-modules. Indeed, $g.\varphi_{\la}(\mu) = \varphi_{\la}(g^{-1}.\mu)$ is $1$ is $\la \leq g^{-1}.\mu$
and $0$ otherwise, while $\varphi_{g.\la}(\mu)$ is $1$ if $g.\la \leq \mu$ and $0$ otherwise. Since
the action of $G$ is order-preserving, the two conditions are equivalent and this proves the claim.
Therefore $c$ induces an isomorphism $\kk M \rtimes \kk G \simeq \kk^{L} \rtimes \kk G$.

When $G$ is a finite and $\kk$ is a field, we have $\kk M \rtimes \kk G \simeq \kk^{L} \rtimes \kk G
\simeq \bigoplus_{X \in E}  \left(\kk^X \rtimes \kk G \right)$
where $E$ is the set of orbits of the action of $G$ on $L$. Each $X$ is a finite, transitive
$G$-set, and therefore $\kk^X \rtimes \kk G \simeq Mat_X(\kk G_0)$,
where $Mat_X(R)$ denotes the $|X|\times |X|$ matrix ring over the ring $R$, and $G_0 < G$ is
the stabilizer of an element of $X$ (see e.g. \cite{CABANESMARIN}, proposition 3.4).
Therefore $\kk M \rtimes \kk G$ is isomorphic to a direct sum of matrix algebras over group algebras
of finite groups. It is thus semisimple if and only if all these group algebras are semisimple.
This is the case as soon as the characteristic of $\kk$ does not divide $|G|$.
Similarly, it is split semisimple if all these group algebras are split semisimple,
and this concludes the proof of the proposition.

\end{proof}

We use this proposition to prove
the following. 

\begin{theorem} \label{thm:semisimple} Let $W$ be a finite Coxeter group. The algebra $\mathcal{C}_W(1)$ is isomorphic to $\kk^{\mathcal{W}} \rtimes \kk W$.
Moreover, if $\kk$ is a field then the following holds.
\begin{enumerate}
\item If the characteristic of $\kk$ does not divide the order of $|W|$, then $\mathcal{C}_W(1)$ is semisimple. If $\kk$ has characteristic $0$, then the algebra $\mathcal{C}_W(\underline{u})$ is generically semisimple,
and $\mathcal{C}_W(\underline{u}) \simeq \kk^{\mathcal{W}} \rtimes \kk W$
for generic $\underline{u}$, up to a finite extension of $\kk$.
\item If moreover the group algebra $\kk N_{W}(W_0)$ of the normalizer of $W_0$ inside $W$
is split semisimple for every reflection subgroup
$W_0$ of $W$, then $\mathcal{C}_W(1)$ is split semisimple.
\end{enumerate}
\end{theorem}
\begin{proof}
We apply the above proposition with $L$ the semilattice made of all the reflection subgroups $\mathcal{W}$,
with $\leq$ denoting the inclusion of reflection subgroups, and the action of $W$
is by conjugation. This proves one part of (1), and the
remaining part is a consequence of Tits' deformation theorem (see e.g. \cite{GECKPFEIFFER}, \S 7.4) and of the fact that
$\mathcal{C}_W(\underline{u})$ is a free module of finite rank over $\kk[\underline{u}]$, by theorem
\ref{thm:main}. Part (2) is the consequence of the proposition above together with the fact that
the stabilizers of the action of $W$ on $\mathcal{W}$ are exactly the normalizers of reflection
subgroups.

\end{proof}

In particular, for $W = \mathfrak{S}_n$, this has the following consequence.
\begin{corollary} \label{cor:SnSplit}  If $W = \mathfrak{S}_n$ and $\kk$ has characteristic
not dividing $n!$, then $\mathcal{C}_W(1)$ is split semisimple over $\kk$.
\end{corollary}
\begin{proof} From the theorem above, we need to prove that, for every reflection
subgroup $W_0$ of $\mathfrak{S}_n$, its normalizer $N_0$ has a split semisimple
group algebra over $\kk$.
Recall that a reflection subgroup $W_0$ of $\mathfrak{S}_n$ naturally corresponds to a partition
$\mathcal{P}$ of $\{1,\dots,n\}$. The normalizer of $W_0$ is easily seen to be the
subgroup of $\mathfrak{S}_n$ stabilizing the partition, and is therefore a direct
product of wreath products of the form $\mathfrak{S}_m \wr \mathfrak{S}_d
= (\mathfrak{S}_m)^d \rtimes \mathfrak{S}_d$ for $md \leq n$. The group algebras
of these groups are split semisimple as soon as they are semisimple (see \cite{JAMESKERBER}, cor. 4.4.9). By Maschke's theorem this holds true as soon as the characteristic of $\kk$ does not
divides $n!$, and this proves the claim.
\end{proof}

We do not know the class of groups,  for which the above corollary holds (in characteristic $0$).
When $W$ is not a Weyl group, the field $\Q$ should of course be replaced by the field of definition
$K = \langle tr(w); w \in W \rangle$. Also, we might want to generalize this statement either
to
$\mathcal{C}_W(1)$ or, more cautiously, to $\mathcal{C}_W^p(1)$ or some $\mathcal{C}_W^R(1)$. The most naive (and vague) question
on Coxeter groups related to this is therefore the following one.
\begin{question}
For which finite Coxeter groups $W$ and which class of reflection subgroups $G$ of $W$ can we expect that the group algebra
$K N_W(G)$ of the normalizer is split semisimple ?
\end{question}

One may wonder whether this is actually true for an arbitrary \emph{reflection subgroup} and the class of \emph{all reflection subgroups}.
%that is whether $\mathcal{C}_W(1)$ is split semisimple for $\kk = K$ and all $W$'s.
A simple and easy-to-visualize counterexample is given by the following construction. Consider
%That this is probably too much to ask can be already hinted from the fact that 
the normalizer
of a 2-Sylow subgroup $S \simeq (\Z/2\Z)^3$ of the symmetry group $W = H_3$ of the icosahedron.
It is a semi-direct product $S \rtimes C_3$, and $S$ is a reflection subgroup --
generated by the reflections around three orthogonal golden rectangles, see figure \ref{fig:icosaedre},
and the
element of order $3$ is
a rotation whose axis goes through the two opposite faces painted in blue.
Therefore this normalizer has (1-dimensional) representations that can
be realized only over $\Q(\zeta_3)$, while the group algebra of $W$ splits only over $\Q(\sqrt{5})$.

Relaxing the first assumption, the next natural question is whether this is actually true for a \emph{Weyl group} (that is, $K = \Q$)
 %arbitrary \emph{reflection subgroup} 
 and again the class of \emph{all reflection subgroups}. A counter-example can be constructed
 %But an actual counterexample for $K = \Q$ (that is, for $W$ a Weyl group) is not so easy to find,
%because it can be checked by computer that the characters of the normalizers
%of the reflection subgroups of all Weyl groups of rank $\leq 6$ are rationally-valued.
in type $E_7$, where there is a 2-reflection subgroup $W_0$ isomorphic to $\Z_2^7$,
whose normalizer $N_0$ has for quotient $N_0/W_0 \simeq \PSL_2(\F_7) \simeq \SL_3(\F_2)$. From the character table
of $\SL_3(\F_2)$ (that can be found e.g. in the ATLAS \cite{ATLAS}) one gets that it
admits
(for example, 3-dimensional) irreducible characters whose values generate $\Q(\sqrt{-7})$,
and therefore the irreducible characters $N_0$ are not all rationally-valued. Interestingly enough, the
reflection subgroups appearing as counter-examples here (for $H_3$ and $E_7$) both arise from the decomposition of $-1 \in W$
as a product of orthogonal reflections, established in \cite{SPRINGERINVOLUTIONS}. For the interested reader,
one can check that, in type $E_7$, we have  $N_0 = \SL_3(\F_2) \ltimes \F_2^7$, and the action of $\SL_3(\F_2)$ on $\F_2^7$
is the permutation representation over $\F_2$ associated to a transitive action of $\SL_3(\F_2)$
on 7 elements. Up to automorphism, there is only one transitive action of $\SL_3(\F_2)$,
and this is its natural action on the seven non-zero elements of $\F_2^3$. I thank R. Stancu for discussions
on this last topic.

%We conclude from this discussion that, in general, the algebras $\mathcal{C}_W(1)$ are \emph{not} split semisimple
%when $\kk = K$, even if we restrict our considerations to the class of Weyl groups.
The next natural question is whether, for all reflection groups, and the class of all parabolic subgroups,
the algebra $K N_W(G)$ is split semisimple, which would imply that $\mathcal{C}_W^p(1)$ is split semisimple for $\kk = K$.
This might be attacked through Howlett's general description of the normalizers of parabolic subgroups (see \cite{HOWLETT}).
Note that the constructions above in type $H_3$ and $E_7$ are not parabolic since they have the same rank as the whole group.

The above discussion on the normalizers motivates to our eyes that the most natural remaining questions on the splitting fields for our
algebras are the following ones.

% that we feel are the more natural are the following ones.
%we are left with the following two natural questions.

\begin{question} Let $W$ be a finite Coxeter group.
\begin{enumerate}
\item Is $\mathcal{C}_W^p(1)$ split semisimple for $\kk = K$ ? At least when $K = \Q$ ?
\item Is there a natural minimal splitting field for $\mathcal{C}_W(1)$ ? Can one characterize it in terms of $W$ ? 
\end{enumerate}
\end{question}

\begin{comment}
The question however remains open for the class of parabolic subgroups, our counterexamples
not being parabolic since they have the same rank as the whole group. In that case,
the conclusion does not appear to be completely obvious from Howlett's general description of their normalizer (see \cite{HOWLETT}). Also, one may be interested to determine a minimal splitting field
for the whole algebra $\mathcal{C}_W(1)$, for each finite Coxeter group $W$.
\end{comment}

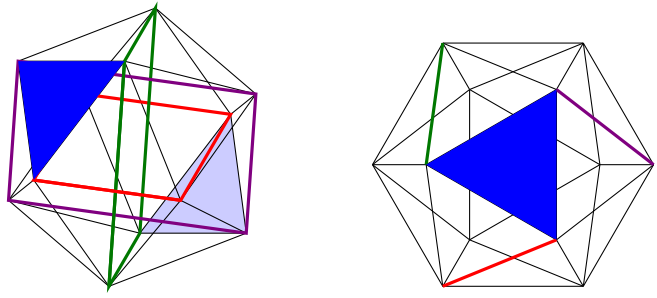
\begin{figure}
\begin{center}
\begin{tikzpicture}
\fill [color=blue!20] (1.3090169944,0.4370160246) -- (0.1019102132,-1.1441228059) -- (1.5161237758,-1.1441228059) -- cycle;
\draw (1.3090169944,0.4370160246) -- (-0.6441228058,0.7071067812);
\draw (1.3090169944,0.4370160246) -- (0.1019102132,-1.1441228059);
\draw (1.3090169944,0.4370160246) -- (0.3090169944,1.8512295871);
\draw (1.3090169944,0.4370160246) -- (1.5161237758,-1.1441228059);
\draw (1.3090169944,0.4370160246) -- (1.6441228059,0.7071067812);
\draw (-0.6441228058,0.7071067812) -- (0.1019102132,-1.1441228059);
\draw (-0.6441228058,0.7071067812) -- (0.3090169944,1.8512295871);
\draw (-0.6441228058,0.7071067812) -- (-1.6441228059,-0.7071067812);
\draw (-0.6441228058,0.7071067812) -- (-1.5161237758,1.1441228059);
\draw (0.1019102132,-1.1441228059) -- (-1.6441228059,-0.7071067812);
\draw (0.1019102132,-1.1441228059) -- (1.5161237758,-1.1441228059);
\draw (0.1019102132,-1.1441228059) -- (-0.3090169944,-1.8512295871);
\draw (0.3090169944,1.8512295871) -- (1.6441228059,0.7071067812);
\draw (0.3090169944,1.8512295871) -- (-1.5161237758,1.1441228059);
\draw (0.3090169944,1.8512295871) -- (-0.1019102132,1.1441228059);
\draw (-1.6441228059,-0.7071067812) -- (-0.3090169944,-1.8512295871);
\draw (-1.6441228059,-0.7071067812) -- (-1.5161237758,1.1441228059);
\draw (-1.6441228059,-0.7071067812) -- (-1.3090169944,-0.4370160246);
\draw (1.5161237758,-1.1441228059) -- (-0.3090169944,-1.8512295871);
\draw (1.5161237758,-1.1441228059) -- (1.6441228059,0.7071067812);
\draw (1.5161237758,-1.1441228059) -- (0.6441228058,-0.7071067812);
\draw (-0.3090169944,-1.8512295871) -- (-1.3090169944,-0.4370160246);
\draw (-0.3090169944,-1.8512295871) -- (0.6441228058,-0.7071067812);
\draw (1.6441228059,0.7071067812) -- (0.6441228058,-0.7071067812);
\draw (1.6441228059,0.7071067812) -- (-0.1019102132,1.1441228059);
\draw (-1.5161237758,1.1441228059) -- (-1.3090169944,-0.4370160246);
\draw (-1.5161237758,1.1441228059) -- (-0.1019102132,1.1441228059);
\draw (-1.3090169944,-0.4370160246) -- (0.6441228058,-0.7071067812);
\draw (-1.3090169944,-0.4370160246) -- (-0.1019102132,1.1441228059);
\draw (0.6441228058,-0.7071067812) -- (-0.1019102132,1.1441228059);
\draw [very thick] [color=red] (1.3090169944,0.4370160246) -- (0.6441228058,-0.7071067812) -- (-1.3090169944,-0.4370160246) -- (-0.6441228058,0.7071067812) -- cycle;
\draw [very thick] [color={rgb,255:red, 0;green, 120;blue, 0}] (0.1019102132,-1.1441228059) -- (0.3090169944,1.8512295871) -- (-0.1019102132,1.1441228059) -- (-0.3090169944,-1.8512295871) -- cycle;
\draw [very thick] [color=violet] (-1.6441228059,-0.7071067812) -- (1.5161237758,-1.1441228059) -- (1.6441228059,0.7071067812) -- (-1.5161237758,1.1441228059) -- cycle;
\draw [very thick] [color={rgb,255:red, 0;green, 120;blue, 0}] (-0.1019102132,1.1441228059)--(-0.3090169944,-1.8512295871);
\draw [very thick] [color=red] (-1.3090169944,-0.4370160246)--(0.6441228058,-0.7071067812);
\fill [color=blue] (-1.5161237758,1.1441228059) -- (-1.3090169944,-0.4370160246) -- (-0.1019102132,1.1441228059) -- cycle;
\end{tikzpicture} \ \ \ \ \ \ \ \ \ \ \ \ 
\begin{tikzpicture}
\draw (-0.5773502691,1) -- (0.934172359,1.6180339888);
\draw (-0.5773502691,1) -- (1.1547005383,0);
\draw (-0.5773502691,1) -- (-0.934172359,1.6180339888);
\draw (-0.5773502691,1) -- (-0.5773502691,-1);
\draw (-0.5773502691,1) -- (-1.868344718,0);
\draw (0.934172359,1.6180339888) -- (1.1547005383,0);
\draw (0.934172359,1.6180339888) -- (-0.934172359,1.6180339888);
\draw (0.934172359,1.6180339888) -- (1.868344718,0);
\draw (0.934172359,1.6180339888) -- (0.5773502691,1);
\draw (1.1547005383,0) -- (1.868344718,0);
\draw (1.1547005383,0) -- (-0.5773502691,-1);
\draw (1.1547005383,0) -- (0.934172359,-1.6180339888);
\draw (-0.934172359,1.6180339888) -- (-1.868344718,0);
\draw (-0.934172359,1.6180339888) -- (0.5773502691,1);
\draw (-0.934172359,1.6180339888) -- (-1.1547005383,0);
\draw (1.868344718,0) -- (0.934172359,-1.6180339888);
\draw (1.868344718,0) -- (0.5773502691,1);
\draw (1.868344718,0) -- (0.5773502691,-1);
\draw (-0.5773502691,-1) -- (0.934172359,-1.6180339888);
\draw (-0.5773502691,-1) -- (-1.868344718,0);
\draw (-0.5773502691,-1) -- (-0.934172359,-1.6180339888);
\draw (0.934172359,-1.6180339888) -- (0.5773502691,-1);
\draw (0.934172359,-1.6180339888) -- (-0.934172359,-1.6180339888);
\draw (-1.868344718,0) -- (-0.934172359,-1.6180339888);
\draw (-1.868344718,0) -- (-1.1547005383,0);
\draw (0.5773502691,1) -- (0.5773502691,-1);
\draw (0.5773502691,1) -- (-1.1547005383,0);
\draw (0.5773502691,-1) -- (-0.934172359,-1.6180339888);
\draw (0.5773502691,-1) -- (-1.1547005383,0);
\draw (-0.934172359,-1.6180339888) -- (-1.1547005383,0);
\draw [very thick] [color=red] (-0.934172359,-1.6180339888) -- (0.5773502691,-1);
\draw [very thick] [color=violet] (1.868344718,0) -- (0.5773502691,1);
\draw [very thick] [color={rgb,255:red, 0;green, 120;blue, 0}] (-0.934172359,1.6180339888) -- (-1.1547005383,0);
\fill [color=blue] (0.5773502691,1) -- (0.5773502691,-1) -- (-1.1547005383,0)-- cycle;
\end{tikzpicture}
\end{center}
\caption{Normalizer of a maximal reflection 2-group in type $H_3$.}
\label{fig:icosaedre}
\end{figure}

\section{Braid image}
\label{sect:braidimage}

In this section we study the image of the (generalized) braid group $B$
inside the algebra $\mathcal{C}_W(\underline{u})$. We let $B^+$
denote the positive braid monoid (or Artin monoid) associated to $W$.

\subsection{Braid morphisms}
\label{sect:braidmorphisms}

\begin{proposition} \label{prop:braidmorphisms} For every collection $(\la_s)_{s \in S} \in (\kk \setminus \{ 0 \})^S$ such that $s \sim t \Rightarrow \la_s = \la_t$,
there exists a morphism $ B^+ \to \mathcal{C}_W(\underline{u})$
defined by $s \mapsto g_s + \la_s g_ se_s$ for $s \in S$. When $\kk$ is a field, it can be extended to a morphism $\Phi_{\la} : \kk B \to \mathcal{C}_W(\underline{u})$ if and only if $\forall s \in S \ \la_s \neq -1 $.
\end{proposition}
\begin{proof}
Let $s,t \in S$, and $m_{st}$ denote the order of $st \in W$.
We have
$$
\begin{array}{cl}
&g_s(1+\la_se_s)g_t(1+\la_te_t)g_s(1+\la_se_s) \dots \\ 
=& \sum g_s (\la_s e_s)^{\eps_1} g_t (\la_t e_t)^{\eps_2} g_s (\la_s e_s)^{\eps_3} \dots \\ 
%= &\sum g_s (\la_s e_s)^{\eps_1} g_t (\la_t e_t)^{\eps_2} g_s (\la_s e_s)^{\eps_3} \dots \\ 
= &\sum (\la_s e_s)^{\eps_1} (\la_{sts} e_{sts})^{\eps_2} g_s  g_t  g_s (\la_s e_s)^{\eps_3} \dots \\ 
= &( \sum  \underbrace{(\la_s e_s)^{\eps_1} (\la_{sts} e_{sts})^{\eps_2} (\la_{ststs} e_{ststs})^{\eps_3} \dots}_{m_{st}}) \underbrace{g_s  g_t  g_s \dots}_{m_{st}}
\end{array}
$$
where the sums are over all the $(\eps_1,\dots,\eps_{m_{st}}) \in \{0,1 \}^{m_{st}}$. By the braid relations inside $\mathcal{C}_W(\underline{u})$ we have $\underbrace{g_s  g_t  g_s \dots}_{m_{st}} = \underbrace{g_t  g_s  g_t \dots}_{m_{st}}$. Finally,
inside the dihedral group $\langle s,t \rangle$, the set of cardinality $m_{st}$ given by
$\{ s, sts, ststs,stststs,\dots \}$ is exactly the union of all the reflections (this is for instance a consequence of the fact
that $\ell(\underbrace{sts\dots}_{m_{st}}) = m_{st}$, see e.g. \cite{BOURB456}  ch. 4 \S 1 no. 4, Lemme 2). From this we get 
$$
\underbrace{g_s(1+\la_se_s)g_t(1+\la_te_t)g_s(1+\la_se_s) \dots }_{m_{st}}
=\underbrace{g_t(1+\la_te_t)g_s(1+\la_se_s)g_t(1+\la_te_t) \dots }_{m_{st}}$$
and this proves the first part. In order to extend this morphism to $B$ it is necessary and sufficient to have $g_s(1+\la_s e_s)$
invertible for all $s \in S$. Since $g_s$ is invertible, this means $(1+ \la_s e_s)$ invertible. Since $e_s^2 = e_s$ and $e_s \not\in \{0, 1\}$
this means $\la_s +1 \neq 0$.
Indeed, we have $(1+\la_s e_s)(1 + \la_s e_s-\la_s -2) = - (\la_s +1)$ hence $(1+\la_s e_s)$
is invertible as soon as $\la_s +1\neq 0$, and conversely $1-e_s$ is not invertible since
$(1-e_s)e_s = 0$.
\end{proof}

\subsection{Description in type $A_1$, and beyond for generic $\la$}

If $W$ has type $A_1$, the algebra $\mathcal{C}_W(\underline{u})$ can be described
by two generators $g,e$ and relations $e^2=e$, $ge = eg$, $g^2 = 1 + (u-1)e(1+g)$.
We know that it is a free module with basis $1,e,g,eg$.
We let $a_0 = (1+g)(1-e)$, $a_1 = e(1+g)$, $a_2 = (g-1)(1-e)$, $a_3 = (g-u)e$.
If $2(u+1)$ is invertible in $\kk$, then $a_0,a_1,a_2,a_3$ is again a basis over $\kk$.
It is made of eigenvectors for $g$ and $e$. The eigenvalues are
$$
\begin{array}{c||c|c|c|c|}
 & a_0 & a_1 & a_2 & a_3 \\
\hline
\hline
e & 0 & 1 & 0 &1 \\
\hline
g & 1 & u & -1 & -1 \\
\hline
eg & 0 & u & 0 & -1
\end{array}
$$
It follows that $g + \la ge$ has eigenvalues 
$1,u(1+\la),-1,-1-\la$. The discriminant of its
characteristic polynomial $(X-1-\la)(X-u(1+\la))(X+1)(X+1+\la)$
is
$$
Q(\la,u) = 4(\la+2)^2( \la u + u-1)^2 (1+u)^2(1+\la)^2(\la u + 1+ u)^2 \la^2.
$$
When this discriminant vanishes, and over a domain, $g+\la ge$ satisfies a cubic relation, because 2 of the 4 eigenvalues are equal. When it is invertible, $g+\la ge$ generates the whole
algebra. As a consequence,
we get for an arbitrary Coxeter group $W$ the following.
\begin{proposition} If the $\la_s,u_s$ as in the
previous proposition are such that $Q(\la_s,u_s)$
is invertible for all $s \in S$, then $\mathcal{C}_W(\underline{u})$
is generated by the $g_s,s \in S$.
\end{proposition}

\subsection{Description in type $A_2$ and beyond for $\la = 0$}

\subsubsection{The cubic Hecke algebra}

For $a,b,c \in \kk$, the $\kk$-algebra $H_3(a,b,c)$ presented
by generators $s,t$ and relations $sts=tst$,
$(s-a)(s-b)(s-c) = (t-a)(t-b)(t-c)=0$ is known
to be a free deformation of the group algebra of the group $\Gamma_3=Q_8 \rtimes (\Z/3\Z)$,
where $Q_8$ is the quaternion group of order $8$ (see \cite{HECKECUBIQUE}). Moreover it is known
to be a symmetric algebra, with explicitely determined Schur elements. 
Specializing $a,b,c$ to $1,-1,u$ we get from \cite{TRBMW} that,
when $\Delta(u) = 6 u(1-u)(u+1)(u^2-u+1)(u^2+u+1)$ is invertible in $\kk$,
then $H_3(1,-1,u)$ is a semisimple algebra, isomorphic to $\kk \Gamma_3$,
possibly after some extension of scalars. We know that $H_3(1,-1,u)$
is a free module of rank $|\Gamma_3| = 24$ and that $\mathcal{C}_{A_2}(\underline{u})$
as rank $30$. Over the field $\kk = \Q(u)$, the image of the natural map 
$H_3(1,-1,u) \to \mathcal{C}_{A_2}(\underline{u})$ can be easily computed,
starting from a basis of $H_3(1,-1,u)$. We get
a vector space of dimension 20.
Therefore, this image is the quotient of $H_3(1,-1,u)$ by one of its three 2-sided
ideals corresponding to its simple modules of dimension 2. This quotient also
appears in the study of the Links-Gould invariant, see \cite{LINKSGOULD}.
This incites to look at skein relation of braid type satisfied by the Links-Gould invariant
on 3 strands. Ishii has established (\cite{ISHII} and also private communication, 2012) that, besides a cubic relation of the form
$(\sigma_i-t_0)(\sigma_i-t_1)(\sigma_i+1)=0$,
the Links-Gould invariant vanishes on the following relation 

$$
\begin{array}{cl}
&s_1 s_2 s_1^{-1} + s_1^{-1} s_2^{-1} s_1 + s_1 s_2 + s_1^{-1} s_2^{-1} + s_2 s_1^{-1} 
+ s_2^{-1} s_1 \\
= &s_1 s_2^{-1} s_1^{-1} + s_1^{-1} s_2 s_1 + s_1 s_2^{-1} + s_1^{-1} s_2 + s_2^{-1} s_1^{-1} + s_2 s_1
\end{array}$$
From explicit calculations inside $H_3(1,t_0,t_1)$ one checks that
this relation is non-trivial in this algebra. Therefore it is a generator of the
simple ideal defining the Links-Gould quotient $LG_3$ in the notations of \cite{LINKSGOULD}.

Another relation communicated by Ishii is the following one.
$$
\begin{array}{cl}
& s_1 s_2^{-1} s_1 - s_2 s_1^{-1} s_2 + t_0 t_1 s_2^{-1} s_1 s_2^{-1} - t_0 t_1
s_1^{-1} s_2 s_1^{-1} \\
=&-(t_0-1)(t_1-1) \left(
s_2^{-1} s_1 - s_1^{-1} s_2 + s_1 s_2^{-1} - s_2 s_1^{-1} + s_1 - s_2 + s_2^{-1} - s_1^{-1}
 \right)
\end{array}
$$
One checks similarly that it is nontrivial in $H_3(1,t_0,t_1)$.
By explicit computations inside $\mathcal{C}_{A_2}(u)$, one checks that both relations
are valid there. For the second one one neeeds to specialize at $\{t_0,t_1 \} = \{1,u \}$.
This proves
\begin{proposition} The two relations above are satisfied inside $\mathcal{C}_{A_k}(u)$
(and therefore inside $Y_{d,k+1}(u)$), for
all $k \geq 2$. Moreover, if $\kk$ is a field and $\Delta(u) \neq 0$,
then the image of $\kk B_3$ inside 
the algebra $\mathcal{C}_{A_2}(u)$ is semisimple, has dimension 20, and can be presented by generators $s_1,s_2$, and the braid relations together with the cubic relation $(s_1 -1)(s_1+1)(s_1-u) = 0$ and \emph{one} of the two relations above.

\end{proposition}

The study of the algebra for a higher number of strands cannot be continued using
the same methods as in \cite{LINKSGOULD}, because the cubic quotient $H_4(1,-1,u)$,
though still being finite dimensional, is conjecturally not semisimple. Indeed,
the Schur elements of a conjectural symmetric trace for $H_4(a,b,c)$ -- as defined and described e.g. in \cite{CHLOUBOOK} --
were computed and included in the development version of the CHEVIE package for GAP3 (see \cite{MICHELCHEVIE}),
and some of them vanish when $(a+b)(a+c)(b+c)=0$.

We computed the dimension of the algebra generated by the braid generators inside $\mathcal{C}_{A_k}(u)$, $k \in \{3,4\}$
for a few rational values of $u$ (including, for $k=4$, $u \in \{ 17,127,217 \}$). We obtained $217$ for $k=3$ and $3364$ for $k = 4$.
This sequence $3,20,217,3364$ of dimensions does not appear for now in Sloane's encyclopaedia of integer sequences, so we could not extrapolate a general formula
from this.

\subsection{Positive representation of the braid monoid for $\la = -1$}
\label{monoidrepresentation}
When $\la = -1$, the images of the Artin generators still satisfy the braid relations,
but they are not invertible anymore. Therefore, they define a representation
of the positive braid monoid, or Artin monoid, that we denote $B^+$.
We denote $b_s = g_s - g_s e_s$
the action of $s \in S$. We have $b_s^3 = b_s$, and a straightforward computation shows that, for all $J \in\mathcal{P}_f(\mathcal{W})$
and $w \in W$, we have
$$
b_s.v_{J,w} = v_{sJs,sw} - v_{sJs \cup \{s \},sw}.
$$
It is remarkable that this action does not depend on the parameters $u_s$ anymore. Moreover,
when $W$ is finite, we can convert it to a linear action with \emph{positive} coefficients, as follows. Composing
through the natural projection $\mathcal{C}_W(\underline{u}) \to \mathcal{C}^{\mathcal{F}_{parab}}_W(\underline{u})$ we get a linear action on a vector space with basis the $x_{W_0,w}$
with $W_0$ a parabolic subgroup of $W$ and $w \in W$. Letting $[J]$ denote the
parabolic closure of $\langle J \rangle$, and $x_{J,w} = x_{[J],w}$
we get that $b_s. x_{[J],w} = x_{[sJs],sw} - x_{[sJs \cup \{ s \}],sw}$
(that is the fixer of the fixed point subspace of $\langle J \rangle$).
The rank of $[J]$ is equal to the codimension of the fixed point space of $\langle J \rangle$.

Note that $b_s.x_{[J],w} = 0$ iff $s \in [sJs]$ iff $s \in [J]$ iff $\rk([J \cup \{s \}]) = \rk([J])$.
Otherwise, $\rk([J \cup \{s \}]) = \rk([J])+1$. Because of this, letting $y_{[J],w} = (-1)^{\rk([J])} x_{[J],w}$
we get the formula
{}
$$
\begin{array}{lcll}
b_s.y_{[J],w} &=& y_{[sJs],sw} + y_{[sJs \cup \{ s \}],sw} \mbox{\ if \ } s \in [J] \\
b_s.y_{[J],w}&= & 0 \mathrm{\ otherwise.} \\
\end{array}
$$
In particular, if $g \in B^+$ is divisible by $s \in S$, then $g.y_{[J],w} = 0$
for all $s \in [J]$. Therefore, one could hope that this representation $g \mapsto b_g$
of $B^+$ is initially injective in the sense given by H\'ee in his analysis of Krammer's
faithfulness criterium (see \cite{HEE}), meaning that
$b_g$ determines the leftmost (or rightmost) simple factor of $g$. This would imply that
the representation $s \mapsto g_s + \la g_s e_s$ of $B$ is faithful, for generic $\la$.
However, this is not the case : in type $A_2$, with generators $s,t$,
a straightforward computation shows that
$b_s^2b_t^3b_s^2 = b_s b_t b_s b_t b_s$ while $s^2t^3s^2$ is divisible
by $s$ and not by $t$ (on both sides), while $ststs = tstts=sttst$
is divisible by $s$ and $t$ on both sides.

Finally, we remark that this representation with positive coefficients cannot be readily transposed to infinite
Coxeter groups. Indeed, although the intersection of all parabolic subgroups containing
a \emph{finitely generated} reflection subgroup of $W$ \emph{is} a parabolic subgroup, and therefore the notion
of parabolic closure remains well-defined, the relation $\rk([J \cup \{r \}]) = \rk([J])+1$ whenever $r \not\in [J]$ fails. The following easy example was communicated to me by T. Gobet. Let $(W,S)$ be an affine Coxeter group of type $\tilde{A}_2$, and $S = \{s,t,u \}$. Let $J = I = \{ s \}$ and $r = tut = utu$. Then $\langle s, t \rangle$
is an infinite dihedral group, whose parabolic closure is $W$, because every proper parabolic subgroup
of $W$ is finite. Therefore $\rk [J \cup \{ r \}] = 2+\rk [J]$ in this case.

\section{Generalization to complex reflection groups}

Let $W < \GL(V)$ be a finite complex reflection group, $\mathcal{R}$ its set of pseudo-reflections,
$\mathcal{W}_{parab}$ the collection of its parabolic subgroups, defined as the fixers of some linear subspace of $V$. We let $\mathcal{A} = \{\Ker(s-1) ,s \in \mathcal{R} \}$ denote the
associated hyperplane arrangement, $X = V \setminus \bigcup \mathcal{A}$ the hyperplane
complement and $B = \pi_1(X/W)$ its braid group. Without loss of generality we may assume that
$\mathcal{A}$ is essential, meaning $\bigcap \mathcal{A} = \{ 0 \}$.
We let $\mathcal{L}$
denote the lattice of the arrangement, formed by the intersections of reflecting hyperplanes.
There is a 1-1 correspondence $\mathcal{L} \to \mathcal{W}$ given by $L \mapsto W_L$
where $W_L = \{ w \in W ; w_{|L} = \Id_{|L} \}$. This bijection is an isomorphism of lattices,
and it is equivariant under the natural actions of $W$.

\subsection{Generalization of $\mathcal{C}_W^p(1)$, and a monodromy representation}
For $\kk$ an arbitrary unital commutative ring,
we let $\kk \mathcal{W}_{parab} = \kk \mathcal{L}$ denote the commutative algebra spanned by a basis of idempotents $e_G, G \in \mathcal{W}$ with relations $e_{G_1} e_{G_2} = e_{[ G_1, G_2 ]}$,
where $[A]$ denotes the parabolic closure of $A$, that is the fixer of the fixed point set of $A \subset W$. Equivalently, it is spanned by idempotents $e_L, L \in \mathcal{L}$ with relations $e_{L_1}e_{L_2}
= e_{L_1 \vee L_2}$, where $e_L = e_{W_L}$. 
In particular, $e_s = e_{Ker(s-1)}$ for all $s \in \mathcal{R}$.
This algebra is naturally
acted upon by $W$, through $w.e_G = e_{wGw^{-1}}$, or equivalently $w.e_L = e_{w(L)}$. We define $\mathcal{C}^p_W(1)$
as the
semidirect product $W \ltimes \kk \mathcal{W}_{parab} \simeq W \ltimes \kk \mathcal{L}$.
It is again acted upon by $W$ through
$w.(w_1.e_G) = (ww_1w^{-1})e_{wGw^{-1}}$. 
Applying proposition 3.7, we have the following analogue of theorem 3.8.
\begin{proposition} \label{prop:CWcrgSS} Let $W$ be a finite complex reflection group, and $\kk$ be a field. The algebra $\mathcal{C}^p_W(1)$
is isomorphic to $\kk^{\mathcal{L}} \rtimes \kk W$. Moreover, if $char.\,\kk$ does not divide $|W|$,
then $\mathcal{C}^p_W(1)$ is semisimple. It is split semisimple as soon as the group algebra $\kk N_W(W_0)$ is split semisimple for all $W_0 \in \mathcal{W}_{parab}$, where $N_W(W_0)$
denotes the normalizer of $W_0$ inside $W$. 
\end{proposition}

We let $\mathcal{T}$ denote the holonomy Lie algebra of the hyperplane
complement $V \setminus \bigcup \mathcal{A}$. Recall from \cite{KOHNONAGOYA} that it is presented by generators $t_H,H \in \mathcal{A}$ and relations $[t_{H_0},t_E] = 0$ for all $H_0 \in \mathcal{A}$ and $E$ a codimension 2 subspace contained in $H_0$ inside the hyperplane lattice (such a subspace
is called a flat), where $t_E = \sum_{H \supset E} t_H$. It is acted upon by $W$ through $w.t_H = t_{w(H)}$.
For $H \in \mathcal{A}$ we let $W_H = \{ w \in W ; w_{|H} = \mathrm{Id}_H \}
\in \mathcal{W}_{parab}$. It is a cyclic group of order $m_H \in \Z_{\geq 2}$. It contains
a unique generator $s_H$ with eigenvalue $\exp(2 \ii \pi/m_H)$, that we call the distinguished
reflection associated to $H \in \mathcal{A}$. We remark that, if $H_2 = w(H_1)$ for some $w \in W$,
then $e_{H_2} = e_{H_1}$ and $w s_{H_1} w^{-1} = s_{H_2}$. 
The following simple fact
will be crucial for us. We state it as a lemma.
\begin{lemma} \label{lem:secomm} Let $H,H_0 \in \mathcal{A}$ and $s \in \mathcal{R}$ such
that $\Ker(s-1) = H$. Then $[se_H,e_{H_0}] = 0$.
\end{lemma}
\begin{proof} We have $se_H.e_{H_0}  = s e_{H \cap H_0}$ and
$e_{H_0}.se_H = s e_{s^{-1}(H_0)} e_H = s e_{s^{-1}(H_0)\cap H}$.
Since $s^{-1}$ acts by the identity on $H$, and $H_0 \cap H \subset H$,
we get $s^{-1}(H_0) \cap H = s^{-1}(H_0) \cap s^{-1}(H) = s^{-1}(H_0 \cap H) = H_0 \cap H$
hence $e_{H_0}.se_H =se_H.e_{H_0}$, which proves the claim. 
\end{proof}

Let us choose for each $H \in \mathcal{A}$ a collection of scalar parameters $\la^{(i)}_H, 0 \leq i < m_H$,
such that the condition $H_2 = w(H_1)$ for some $w \in W$ implies $\la^{(i)}_{H_1} = \la^{(i)}_{H_2}$
for all $0 \leq i < m_{H_1} = m_{H_2}$.

\begin{proposition} \label{prop:repholo} There exists a (necessarily unique) morphism of Lie algebras
$\varphi : \mathcal{T} \to \mathcal{C}_W(1)$ satisfying
$$
t_H \mapsto \left( \sum_{0 \leq i < m_H} \la^{(i)}_H s_H^i \right) e_H
$$
This morphism is $W$-equivariant.
\end{proposition}
\begin{proof} Let $E$ be a codimension 2 flat, and $H_0 \in \mathcal{A}$ such that $E \subset H_0$.
Let $s \in \mathcal{R}$ such that $\Ker(s-1) = H_0$. It is enough to prove that $s e_{H_0}$
commutes with the image of $t_E$ for all such $s$. We have
$$
s e_{H_0} \left( \sum_{H \supset E} \left( \sum_{0 \leq i < m_H} \la^{(i)}_H s_H^i \right) e_H \right) =
 \sum_{H \supset E}  \sum_{0 \leq i < m_H} \la^{(i)}_H  s s_H^i e_{s_H^i(H_0)}  e_H 
 $$
 {}
 $$
=  \sum_{0 \leq i < m_{H_0}} \la_{H_0}^{(i)} s s_{H_0}^i e_{H_0} +  \sum_{\stackrel{H \supset E}{H \neq H_0}}  \sum_{0 \leq i < m_H} \la^{(i)}_H  s s_H^i e_{s_H^i(H_0)}  e_H 
$$
and 
$$
\left( \sum_{H \supset E} \left( \sum_{0 \leq i < m_H} \la^{(i)}_H s_H^i \right) e_H \right) s e_{H_0}  =
 \sum_{H \supset E}  \sum_{0 \leq i < m_H} \la^{(i)}_H   s_H^i s e_{s^{-1}(H)}  e_{H_0} 
 $$
 {}
 $$
 = \sum_{0 \leq i < m_{H_0}} \la^{(i)}_{H_0}   s_{H_0}^i s   e_{H_0} +
  \sum_{\stackrel{H \supset E}{H \neq H_0}}  \sum_{0 \leq i < m_H} \la^{(i)}_H   s_H^i s e_{s^{-1}(H)}  e_{H_0}.
$$
We notice that $s s_{H_0}^i = s_{H_0}^i s$ for all $i$, since $s \in W_{H_0} = \langle s_{H_0} \rangle$.
Moreover, if $H \neq H_0$, then $s^{-1}(H) \neq H_0$ and $s_H^i(H_0) \neq H$,
and therefore $e_{s^{-1}(H)}  e_{H_0}  = e_{s_H^i(H_0)} e_H = 
e_E$. Therefore, it is sufficient to prove that
$$
\sum_{\stackrel{H \supset E}{H \neq H_0}}  \sum_{0 \leq i < m_H} \la^{(i)}_H  s s_H^i
 = 
s \left(
\sum_{\stackrel{H \supset E}{H \neq H_0}}  \sum_{0 \leq i < m_H} \la^{(i)}_H   s_H^i s \right)s^{-1}
$${}$$
=
\sum_{\stackrel{s(H) \supset E}{s(H) \neq H_0}}  \sum_{0 \leq i < m_H} \la^{(i)}_H   s_H^i s 
=
  \sum_{\stackrel{H \supset E}{H \neq H_0}}  \sum_{0 \leq i < m_H} \la^{(i)}_H   s_H^i s
  $$
  which holds true. The $W$-equivariance is clear.
\end{proof}

Now assume that the $\la_{H,i}$ are complex numbers, and that $\kk$ contains $\C[[h]]$ as a subring.
We consider the $\mathcal{C}_W(1)$-valued 1-form
$$
\om = \frac{1}{\ii \pi} \sum_{H \in \mathcal{A}} h\varphi(t_H) \om_H
$$
with $\om_H$ is the canonical logarithmic form around $H$ defined as $\mathrm{d} \alpha_H/\alpha_H$
for some linear form $\alpha_H \in V^*$ with kernel $H$.
Proposition \ref{prop:repholo} states that $\om \in \Omega^1(X) \otimes \mathcal{C}_W(1)$
is integrable and $W$-invariant. By monodromy we get a representation $B \to \mathcal{C}_W(1)$,
where the image of a braided reflection $\sigma$ associated to $s=s_{H}^k \in \mathcal{R}$
is conjugated to $s_H^k \exp(h \varphi(t_H))$, that is
$$
s_H^k \exp\left( h \left( \sum_{0 \leq i < m_H} \la^{(i)}_H s_H^i  \right) e_H\right).
$$
The algebra $\mathcal{C}_{W_H}(1)$ can be decomposed as $(\kk C) (e_H-1) \oplus (\kk C)e_H$
where $C = \langle s_H \rangle$, since $e_H$ is a central idempotent inside $\mathcal{C}_{W_H}(1)$.
This proves that $\sigma$ is semisimple with eigenvalues $1, \zeta^k$ together
with the $\zeta^{kr}\exp(h\sum_{0 \leq i < m_H} \la_H^{(i)} (\zeta)^{ri})$ for $0 \leq r < m_r$
for $\zeta = \exp(2 \ii \pi/m_H)$.

\subsection{Generalization of $\mathcal{C}_W^p$}

Let $a_{H,i}, 0 \leq i < m_H$ be a collection of parameters in $\kk$ such that
$a_{w(H),i} = a_{H,i}$ for all $H \in \mathcal{A}$, $w \in W$, and $a_{H,0}\in \kk^{\times}$.
The natural action of $W$ on $\mathcal{L}$ can be considered as an action of $B$,
through the natural morphism $\pi : B \onto W$, so that the
relation $g e_L = e_{\pi(g)(L)} g$ for all $g \in B$, $L \in \mathcal{L}$ holds inside $B \ltimes \kk \mathcal{L}$.

\begin{definition} \label{def:CWPcrg} The algebra $\mathcal{C}_W^p(\underline{a})$ is defined
as the quotient of the group algebra $B \ltimes \kk \mathcal{L}$ by the relations
$$\sigma^{m_H} = 1 + e_H (\sum_{k=0}^{m_H-1} a_{H,k}\sigma^k - 1)$$ for any braided reflection
$\sigma$ associated to $s_H \in \mathcal{R}$ with $H \in \mathcal{A}$.

\end{definition}

We remark that the relation of the semidirect product is the consequence
of the more elementary relations $g e_H = e_{\pi(g)(H)} g$
for $H \in \mathcal{A}$ and $g$ running among a generating system of $B$. Similarly,
since all braided reflections corresponding to some $H \in \mathcal{A}$ are conjugated
one to the other by an element of $B$, using the relations of the semidirect product, the defining
relations of $\mathcal{C}_W^p(\underline{a})$ can be replaced
by imposing the same relations but only for one braided reflection per conjugacy
class of hyperplanes. Since $B$ is known to be finitely generated (and even finitely presentable), these remarks prove
the following
\begin{proposition} {\ }
\begin{enumerate}
\item The algebra $\mathcal{C}^p_W(\underline{a})$ is finitely generated (and even finitely presentable) as an algebra.
\item If $W$ is a finite Coxeter group with generating system $S$, and if $a_{\Ker(s-1),0} = u_s$, $a_{\Ker(s-1),1}=u_s-1$,
then $\mathcal{C}^p_W(\underline{a}) = \mathcal{C}^p_W(\underline{u})$.
\end{enumerate}

\end{proposition}

The Hecke algebra $H_W(\underline{a})$ of $W$ is defined as the
quotient of $\kk B$ by the relations $\sigma^{m_H} = \sum_{k=0}^{m_H-1}a_{H,k}\sigma^k $
for any braided reflection
$\sigma$ associated to $s_H \in \mathcal{R}$ with $H \in \mathcal{A}$.
We remark that the algebra $\kk \mathcal{L}$ admits an augmentation map $\kk \mathcal{L} \to \kk$
defined by $e_L \mapsto 1$, which is split through $1 \mapsto e_W = e_{\{ 0 \}}$.
From this the following is immediate.
\begin{proposition} The maps $\kk B \to H_W(\underline{a})$ and $\kk \mathcal{L} \mapsto \kk$
together induce an algebra morphism $B \ltimes \kk \mathcal{L} \to H_W(\underline{a}) \otimes_{\kk} \kk = H_W(\underline{a})$. It factorizes through a surjective algebra morphism
$\mathfrak{p} : \mathcal{C}_W^p(\underline{a}) \onto H_W(\underline{a})$. The morphism $\kk B
\to \mathcal{C}_W^p(\underline{a})$ induced by $b \mapsto b e_W$ for $b \in B$
factors through a splitting $\mathfrak{q} : H_W(\underline{a}) \to \mathcal{C}_W^p(\underline{a})$
of $\mathfrak{p}$.
\end{proposition}

\begin{proposition} \label{prop:monomorphsurj}
Assume that $\kk = \C[[h]]$, and $v_{H,i} = \zeta_{m_H}^i \exp(h  \tau_{H,i})$ for some
$\tau_{H,i} \in \C$.Then there exists a surjective morphism $\mathcal{C}_W^p(\underline{a}) \to
\mathcal{C}_W^p(1)$, where 
$$X^{m_H} - \sum_{0 \leq i <m_H} a_{H,i} X^i = \prod_{0 \leq i < m_H} (X- v_{H,i}).
$$
\end{proposition}
\begin{proof}
For instance by invertibility of the Vandermonde determinant, one can
find complex scalars $\la_{H,i}$ such that 
$\sum_{0 \leq i < m_H} \la_H^{(i)} (\zeta_H)^{ri}) = \tau_{H,r}$ for $0 \leq r < m_H$, with
$\zeta_H = \exp(2 \ii \pi/m_H)$. We consider the monodromy morphism $R : \kk B \to \mathcal{C}_W^p(1)$
constructed above. The image of a braided reflection $\sigma$ associated so $s_H$
has eigenvalues $1, \zeta$ and $\zeta^{r} \exp( h\tau_{H,r}) = v_{H,r}$. For instance
by using Chen's iterated integrals, we notice that, for $b \in B$,
$R(b)$ has the form $\beta M(h)$,
where $\beta \in W \subset \mathcal{C}_W^p(1)$ is the image of $b \in B$ under the natural map
$B \onto W$, and $M(h) \in A_0[[h]]$, where $A_0$ is the subalgebra of $\mathcal{C}_W^p(1)$
generated by the $s e_H$, for $s \in \mathcal{R}$ and $H = \Ker(s-1)$. Lemma \ref{lem:secomm}
implies that $A_0$ commutes with all $e_L, L \in \mathcal{L}$. Therefore,
we have $R(b) e_L R(b)^{-1} = \beta M(h) e_L M(h)^{-1} \beta^{-1} = \beta e_L \beta^{-1}
= e_{\beta(L)} = e_{b(L)}$. This proves that $R$ can be extended to
a morphism $B \ltimes \kk \mathcal{L} \to \mathcal{C}_W^p(1)$ through $b \otimes e_L \mapsto
R(b) \otimes e_L$.

It remains to prove that the defining relations of $\mathcal{C}_W^p(\underline{a})$ are satisfied.
Let $H \in \mathcal{A}$, $s = s_H$ and $\sigma$ a braided reflection associated to them.
For short, let $S = R(\sigma)$ and $S_0 = s \exp(h \varphi(t_H))$. We have $S = P s \exp(h \varphi(t_H)) P^{-1}$
for some $P \in A_0[[h]]$. Since $\varphi(t_H)$ commutes with $A_0$,
we get $S^m = P \exp(m h \varphi(t_H))P^{-1}
= 1+ P (\exp(m h \varphi(t_H))-1)P^{-1}$. We have $\exp(m h \tau_{H,r}) = v_{H,r}^{m_H}$
and $v_{H,r}^{m_H} - \sum_i a_{H,i} v_{H,r}^i = \prod_i (v_{H,r}-v_{H,i})=0$.
Now, the compared spectrum of the elements in play is as follows
$$
\begin{array}{|c||c|c|c|}
\hline 
S_0= s \exp(h \varphi(t_H)) & 1 & \zeta & v_{H,r} \\
\hline
e_H & 0 & 0 & 1 \\
\hline
S_0e_H & 0 & 0 & v_{H,r} \\
\hline
S_0^m & 1 & 1 & v_{H,r}^m \\
\hline 
S_0^m-1 & 0 & 0 & v_{H,r}^m -1 \\
\hline
e_HS_0^i & 0 & 0 & v_{H,r}^i \\
\hline
\end{array}$$
Therefore, we have $S_0^m -1 = e_H((\sum_{0 \leq i < m_H} a_{H,i} S_0^i) - 1)$
hence $S^m  
= 1 + Pe_H((\sum_{0 \leq i < m_H} a_{H,i} S_0^i) - 1)P^{-1}
= 1 + e_H((\sum_{0 \leq i < m_H} a_{H,i} (PS_0P^{-1})^i) - 1)
= 1 + e_H((\sum_{0 \leq i < m_H} a_{H,i} S^i) - 1)
$
and this proves the claim.

\end{proof}

We remark that proposition \ref{prop:braidmorphisms} admits no direct
generalization to the complex reflection groups setting, namely there
is not in general a 1-parameter family of morphisms $B \to \mathcal{C}_W^p(\underline{a})$
of a similar form. Indeed, let us consider for $W$ the group generated by 2-reflections
called $G_{12}$ in the Shephard-Todd classification. Its braid group has the presentation
$\langle s,t,u \ | \ stus = tust = ustu \rangle$ and $W = B/\langle s^2,t^2,u^2 \rangle$.
Letting $e_x \in \mathcal{C}_W^p(\underline{a})$ denote the idempotent associated
to the hyperplane $\Ker( x - 1)$, for $x \in W$ a reflection, one can check that 
there can be
% maps $y \mapsto y+ \la e_y y$, for $y \in \{ s, t, u \}$, define
a morphism $B \to \mathcal{C}_W^p(\underline{a})$ 
satisfying $y \mapsto y+ \la e_y y$, for $y \in \{ s, t, u \}$
only if the
4 reflecting hyperplanes associated to the reflections $\{ s, sts, stuts, stusuts \}$
are the same as the ones associated to the reflections $\{ t,tut,tusut,tustsut \}$ (equivalently, that these
two sets of 2-reflections are equal). One readily
checks that this does \emph{not} hold.

\subsection{An extended freeness conjecture}
For a $W$-orbit of hyperplanes $c$, the order $m_H$ of $W_H$ for $H \in c$
depends only on $c$. Therefore, we can denote it $m_c$, and define
a generic ring $R_W = \Z[ (a_{c,i},a_{c,0}^{-1}]$ for $c \in \mathcal{A}/W$ and $0 \leq i < m_c$.
The \emph{generic algebra} $\mathcal{C}_W^p$ is defined over the ring $\kk = \Z[ (a_{c,i},a_{c,0}^{-1}]$ as in definition \ref{def:CWPcrg} by letting $a_{H,i} = a_{c,i}$ if $H \in c$.

\begin{proposition} \label{prop:crit1freeness} If the algebra $\mathcal{C}^p_W$
is spanned by $|W|.|\mathcal{L}|$ elements as a $R_W$-module, then it is a free $R_W$-module of rank $|W|.|\mathcal{L}|$.
\end{proposition}
\begin{proof} The proof follows exactly the same lines as in \cite{BMR} (proof of theorem 4.24),
see also \cite{CYCLO} proposition 2.4, the `monodromic' ingredient being given by
proposition \ref{prop:monomorphsurj} above. It is left to the reader.
\end{proof}

We consider $\mathcal{C}^p_W(\underline{a})$ as a $\kk B$-module.
As a $\kk B$-module, it is generated by the $e_L, L \in \mathcal{L}$. Let
$\mathcal{E}_L = \sum_{F \subset L} (\kk B).e_F$,
$\mathcal{E}'_L = \sum_{\stackrel{F \neq L}{F \subset L}} (\kk B).e_F$
and $\overline{\mathcal{E}}_L = \mathcal{E}_L/\mathcal{E}'_L$.

\begin{lemma} If each $\overline{\mathcal{E}}_L$ is spanned as a $\kk$-module
by $|W|$ elements  of the form $b.e_L, b \in B$, then $\mathcal{C}^p_W(\underline{a})$
is spanned by $|W|.|\mathcal{L}|$ elements, and therefore it is a free $\kk$-module of rank $|W|.|\mathcal{L}|$.
\end{lemma}
\begin{proof}
Assume that, for each $L$, we have elements $b_{L,w}, w \in W$
such that $\overline{\mathcal{E}}_L$ is spanned by the $b_{L,w}.e_L$.
We shall prove that $\mathcal{C}_W^p(\underline{a})$ is spanned
by the $b_{L,w}.e_L$ for $L \in \mathcal{L}, w \in W$. Since $\mathcal{C}_W^p(\underline{a})$
is generated as a $\kk B$-module by the $e_L, L \in \mathcal{L}$,
it is spanned as a $\kk$-module by the $b e_L, L \in \mathcal{L}$. Therefore,
it is sufficent to prove that such a $b e_{L_0}$ is a linear combination of the 
$b_{L,w}.e_L,L \in \mathcal{L}$. We prove this by induction on $L_0$
with respect to the well-ordering provided by the lattice $\mathcal{L}$.
If $L_0 = \{ 0 \}$, then $b.e_L = b.e_W \in \mathcal{E}_L = \overline{\mathcal{E}}$ and
we have the conclusion by assumption. If not, we know that
there exists scalars $\alpha_{L_0,w}, w \in W$ such that $x = b.e_{L_0} - \sum_{w \in W} \alpha_{L_0,w} b_{L_0,w}.e_{L_0}  \in \mathcal{E}'_{L_0}$. By the induction assumption we can
write $x$ as a linear combination of the $b_{L,w}e_L$ for $L \subsetneq L_0$, and
therefore $b.e_{L_0}$ as a linear combination of the $b_{L,w}e_L$ for $L \subset L_0$,
and this proves the claim.
\end{proof}

We notice that the action of $\kk B$ on $\overline{\mathcal{E}}_{\{ 0 \}} = \mathcal{E}_{\{ 0 \}}$
factorizes through $H_W(\underline{a})$, and therefore $\overline{\mathcal{E}}_{\{ 0 \}}$
is spanned by $|W|$ elements if and only if the BMR freeness conjecture is true for $W$.
We also notice that the action of $\kk B$ on $\overline{\mathcal{E}}_{V}$
factorizes through the regular representation of $\kk W$, hence $\overline{\mathcal{E}}_{V}$
is clearly spanned by $|W|$ elements.

In this way, the presumed fact that each $\overline{\mathcal{E}}_L$ is spanned
by $|W|$ elements appears as an intermediate between the trivial fact that $\kk W$
has this property and the BMR freeness conjecture that $H_W$ is spanned by $|W|$ elements.
For a given $L \neq \{ 0 \}$, and if true, it should be easier to prove than the freeness conjecture for $H_W$, since, at each stage, the relation $g_s^m = \dots$  to be used can be either the complicated (Hecke) one or the trivial one ($g_s^m = 1$). However, it does not seem to readily follow from
it, and therefore we propose it as a (a priori stronger) conjecture.

\begin{conjecture} \label{conj:freeness} (extended freeness conjecture) The algebra $\mathcal{C}_W^p(\underline{a})$ is a free $\kk$-module
of rank $|W|.|\mathcal{L}|$. Moreover, each module $\overline{\mathcal{E}_L}, L \in \mathcal{L}$,
is spanned by $|W|$ elements of the form $b_{L,w}.e_L, w \in W$, with $b_{L,w} \in B$ mapping to $w \in W$ under the natural map $B \onto W$.
\end{conjecture}

If $\mathcal{C}^p_W$ is a $R_W$-module of rank $|W|. |\mathcal{L}|$, then
it is a free deformation of the algebra $\mathcal{C}_W(1)$, which is semisimple for $\kk= \Q$
by proposition \ref{prop:CWcrgSS}. Therefore, Tits' deformation theorem  (see e.g. \cite{GECKPFEIFFER}, \S 7.4) and
proposition \ref{prop:CWcrgSS} imply the following,
where $K_W$ denotes a
field containing $R_W$.

\begin{proposition} If the extended conjecture is true, then $\mathcal{C}_W^p \otimes_{R_W} K_W$
is semisimple. If moreover $K_W$ is algebraically closed, then 
$\mathcal{C}_W^p \otimes_{R_W} K_W \simeq \mathcal{C}_W^p(1) \otimes K_W
\simeq K_W^{\mathcal{L}} \rtimes K_W W$.
\end{proposition}

If $W$ has rank 2 and the BMR freeness conjecture is true for $W$, the proof is reduced
to the consideration of the $\overline{\mathcal{E}}_H$ for $H \in \mathcal{A}$.
Since 
$g b e_L g^{-1} = gbg^{-1} e_{\pi(g)(L)}$ for all $g \in B$,
we moreover need to consider only one hyperplane per $W$-orbit.

\subsection{The case of $G_4$}

The smallest non-trivial example of an irreducible non-real complex reflection group
outside the infinite series of monomial groups
is the group $Q_8 \rtimes \Z_3$ denoted $G_4$ in Shephard-Todd notation.
It is also the group for which the original BMR freeness conjecture has had, so far, the
more topological applications (see e.g. \cite{LINKSGOULD, TRBMW}).
In this case $B = \langle s_1,s_2 \ | \ s_1 s_2 s_1 = s_2 s_1 s_2 \rangle$ is the Artin
group of type $A_2$ (a.k.a. the braid group on 3 strands) and $W$ is the quotient of $B$
by the relations $s_1^3 = s_2^3 = 1$. A proof of the original BMR freeness conjecture
for this case can be found for instance in \cite{HECKECUBIQUE}. 

We let $\mathcal{B} = \{ 1, s_1^{\eps}, s_1^{\alpha} s_2^{\eps} s_1^{\beta}, s_2^{-1}s_1 s_2^{-1} s_1^{\alpha} \}$ where $\alpha,\beta \in \{-1,0,1 \}$, $\eps \in \{-1, 1 \}$. We have $\# \mathcal{B} = 24 = |W|$,
and we want to prove that $\mathcal{B}.e_1 \subset E = \overline{\mathcal{E}_1}$
is a $\kk B$-submodule. This will prove $\mathcal{B}.e_1 = E$ since $1 \in \mathcal{B}$.
Since $B$ is generated as a group by $s_1,s_2$ it is sufficient to prove that $s_i be_1 \subset \mathcal{B}.e_1$ for $i \in \{ 1, 2 \}$ and $b \in \mathcal{B}$. We let $e_i = e_{\Ker(s_i-1)}$
pour $i \in \{1,2,3,4 \}$ by letting
$s_3 = s_1 s_2 s_1^{-1}$, $s_4 = s_2 s_1 s_2^{-1}$.
Inside $E$ we have $e_i e_j = \delta_{ij}e_i = \delta_{ij}e_j$. 
We have $s_i e_j =e_{\sigma_i(j)} s_i$ where $\sigma_1 = (2,3,4)$ and $\sigma_2 = (1,4,3)$.
By definition we have $s_i^3 = 1 + e_i P_i = 1+P_i e_i$ for $P_i = a s_i^2 + b s_i + c-1$,
and this implies $s_i^2 = s_i^{-1} + e_i Q_i = 1+Q_i e_i$ with $Q_i = a s_i + b + (c-1)s_i^{-1}$.

\begin{lemma}{\ } \label{lem:techG4}
\begin{enumerate}
\item For all $m \in \Z$, $b \in B$, $i \in \{1,2\}$, we have $s_i^m b e_1 \in \langle s_i^{\alpha} b e_1,
\alpha \in \{-1,0,1 \} \rangle$.
\item For all $\alpha,\beta \in \Z$ we have
$s_2^{\alpha} s_1 s_2^{-1} s_1^{\beta} e_1 \in \langle \mathcal{B}e_1 \rangle$.
\item For all $x \in \langle s_1^{m}, m \in \Z  \rangle$ we have
$(s_2 s_1^{-1} s_2)s_1 x e_1 = s_2^{-1} s_1 s_2^{-1} x e_1$
\item For all $x \in \langle s_1^{m}, m \in \Z  \rangle$ we have $(s_2^{-1} s_1 s_2^{-1})x e_1 = x (s_2^{-1} s_1 s_2^{-1}) e_1$.
\end{enumerate}
\end{lemma}
\begin{proof}
Since $s_i^2 = s_i^{-1} + e_i Q_i$ we have
$s_i^2b e_1 = s_i^{-1}b e_1 +  Q_i e_i b e_1 = s_i^{-1}b e_1 +  Q_i e_{\beta(i)} e_1
=  s_i^{-1}b e_1 +  Q_i \delta_{\beta(i),1} e_1$ for some $\beta \in \mathfrak{S}_4$.
This proves $s_i^2 b e_1 \in  \langle s_i^{\alpha} b e_1,
\alpha \in \{-1,0,1 \} \rangle$ hence (1).

We now prove (2). By item (1) we can assume $\alpha,\beta \in \{-1,0,1 \}$. If $\alpha \in \{ 0, -1 \}$,
then $s_2^{\alpha} s_1 s_2^{-1} s_1^{\beta} e_1  \in \mathcal{B} e_1$ by definition.
If $\alpha=1$, then $s_2^{\alpha} s_1 s_2^{-1} s_1^{\beta} e_1 = (s_2 s_1 s_2^{-1}) s_1^{\beta} e_1
= s_1^{-1} s_2 s_1 s_1^{\beta} e_1 =  s_1^{-1} s_2 s_1^{\beta+1} e_1$ by the braid relation. This element belongs to $\langle \mathcal{B} e_1 \rangle$ by (1), and this proves (2).

We now prove (3). We have $s_2 (s_1^{-1} s_2s_1) x e_1=
s_2^2 s_1s_2^{-1} x e_1= (s_2^{-1} - Q_2e_2) s_1 s_2^{-1} x e_1
= s_2^{-1}  s_1 s_2^{-1} xe_1  - Q_2 e_2s_1 s_2^{-1} x e_1$. Now $e_2 s_1 s_2^{-1} = s_1 s_2^{-1} e_3$ and $e_3s_1^{\alpha}e_1 = 0$ for all $\alpha$, hence $ Q_2 e_2s_1 s_2^{-1} x e_1 = Q_2 s_1 s_2^{-1}x e_3 e_1$.

We now prove (4). We first assume $x = s_1^{-1}$.
We have $s_1 s_2^{-1} s_1^{-1}e_1 = (s_1^{-2} - A e_1) s_2^{-1} s_1^{-1} e_1 $for some $A \in \kk B$.
Since $e_1 s_2^{-1} s_1^{-1} e_1 = 0$ we get $s_1 s_2^{-1} s_1^{-1}e_1 = s_1^{-2}  s_2^{-1} s_1^{-1} e_1= s_1^{-1} (s_1^{-1} s_2^{-1} s_1^{-1}) e_1 =  s_1^{-1} s_2^{-1} s_1^{-1} s_2^{-1} e_1
$. Thus $s_2^{-1} . s_1 s_2^{-1} s_1^{-1}e_1 = s_2^{-1}.s_1^{-1} s_2^{-1} s_1^{-1} s_2^{-1} e_1
= s_1^{-1}s_2^{-1} s_1^{-1} s_1^{-1} s_2^{-1} e_1$ by the braid relation
Now, $s_1^{-2} s_2^{-1} e_1 = s_1 s_2^{-1} e_1$ since $e_1 s_2^{-1} e_1 = 0$
hence $s_2^{-1} . s_1 s_2^{-1} s_1^{-1}e_1 =  s_1^{-1}s_2^{-1} s_1 s_2^{-1} e_1$
and we have proven (4) for $x = s_1^{-1}$. This implies that (4) holds true for every
power of $s_1^{-1}$, whence the claim.

\end{proof}

We now prove that $s_i b e_1 \in \langle \mathcal{B} e_1 \rangle$ for all $b \in \mathcal{B}$, $i \in \{1,2 \}$. 
For $b \in \{ 1 , s_1^{\eps}  \}$ this is clearly true. Now assume
$b = s_1^{\alpha} s_2^{\eps} s_1^{\beta}$. If $i = 1$
this is an immediate consequence of the definition of $\mathcal{B}$ and of lemma \ref{lem:techG4} (1).
We consider $s_2 b e_1$. If $\eps = -1$, then $s_2 b e_1 = s_2 s_1^{\alpha} s_2^{-1} s_1^{\beta}e_1
 = s_1^{-1} s_2^{\alpha} s_1 s_1^{\beta} e_1$ by the braid relation, 
 hence $s_2 b e_1  \in \langle \mathcal{B} e_1 \rangle$. Therefore we can assume $\eps = 1$.
If $\alpha = 0$, then $s_2 b e_1 = s_2^{2} s_1^{\beta} e_1
\in \langle \mathcal{B} e_1 \rangle$ by lemma \ref{lem:techG4} (1). If $\alpha = 1$,
we have $s_2 b e_1 = (s_2 s_1 s_2) s_1^{\beta} e_1
= (s_1 s_2 s_1) s_1^{\beta} e_1 \in \langle \mathcal{B} e_1 \rangle $ by (1).
If $\alpha = -1$, then $s_2 b e_1 = s_2 s_1^{-1} s_2 s_1^{\beta} e_1
 = s_2 s_1^{-1} s_2 s_1 s_1^{\beta-1} e_1
  = s_2^{-1} s_1 s_2^{-1} s_1^{\beta-1} e_1$ by (3), hence $s_2 b e_1 \in \langle \mathcal{B} e_1 \rangle$.
This concludes the case $b = s_1^{\alpha} s_2^{\eps} s_1^{\beta}$. We now assume $b =  s_2^{-1} s_1 s_2^{-1} s_1^{\alpha}$. We have $s_2 b e_1 = s_1 s_2^{-1} s_1^{\alpha} e_1 \in \mathcal{B}$,
and $s_1 b e_1 = s_1 s_2^{-1} s_1 s_2^{-1} s_1^{\alpha} e_1 = s_2^{-1} s_1 s_2^{-1} s_1^{\alpha+1}e_1$
by (4), which belongs to $\langle \mathcal{B} e_1 \rangle$ by (1). This proves conjecture \ref{conj:freeness} for $W = G_4$.

\subsection{An extended Ariki-Koike algebra}

Let $W = G(d,1,n)$ be the group of $n\times n$ monomial matrices with entries in $\mu_d(\C) \cup \{ 0 \}$. In this case $H_W$ is known
as the Ariki-Koike algebra, and $B$ is the Artin group of type $B_n/C_n$, with generators $t = a_1,a_2,\dots, a_n$.
The images of $a_i$ and $a_j$ inside $W$ satisfy an Artin (braid) relation of length $4$ if $\{i,j\} = \{1,2 \}$, $2$ if $|i-j|\geq 2$,
and $3$ otherwise.
If we abuse notations by letting $e_b, b \in B$ be equal to $e_{\Ker(\beta-1)}$ for $\beta \in W$
the image of $b$, we have inside $\mathcal{C}_W^p$
the relations $t^d = 1 + (q-1) e_tP(t)$ for some polynomial $P$ of degree at most $d-1$,
and $a_i^2 = 1+(q-1)(a_i+1)e_{a_i}$ for $i \geq 2$. We adapt the arguments of \cite{ARIKIKOIKE}
to prove conjecture \ref{conj:freeness}.

First of all, we let $t_1 = t$, $t_i = a_i t_{i-1} a_i$. There is a classical injective morphism $B \to B_{n+1}$, where $B_{n+1}$ is the usual braid group on $n+1$ strands, given by $t \mapsto \sigma_1^2$,
$a_i \mapsto \sigma_{i+1}$ for $i \geq 2$, where $\sigma_1,\dots,\sigma_{n+1}$ denote the classical Artin generators. Under this map, each $t_i$ is mapped to $\delta_{i+1} = \sigma_{i+1}\sigma_i
\dots \sigma_2 \sigma_1^2 \sigma_2 \dots \sigma_{i+1}$, and $\delta_i = z_{i+1}/z_i$ where
$z_i = (\sigma_1\dots\sigma_{i-1})^i$ is the canonical generator of $Z(B_i)$. From this, we have the
following relations inside $B$, and therefore inside $\mathcal{C}_W^p$:
\begin{enumerate}
\item For all $i,j$ with $j \not\in \{i,i+1\}$ we have $a_j t_i = t_i a_j$
\item For all $i,j$ we have $t_i t_j = t_j t_i$
\item For all $i$, we have $a_i t_{i-1} t_i = t_{i-1}t_i a_i$.
\end{enumerate}
Let $E$ denote the (commutative) subalgebra of $\mathcal{C}_W^p$
generated by the $e_H, H \in \mathcal{A}$. 
Note that $E b \subset bE$ for all $b \in B$. The above equalities moreover imply that
the $t_i, i \geq 1$ generate a commutative subalgebra of $\mathcal{C}_W^p$.
We prove the following lemma.
\begin{lemma} \label{lem:techAK} For all  $k \geq 1$,
\begin{enumerate}
\item For all $i \geq 2$, $a_i t_i^k \in t_{i-1}^k a_i E + \sum_{j=1}^k t_{i-1}^{j-1} t_i^{k+1-j} E$
\item For all $i \geq 1$, $a_{i+1} t_i^k \in t_{i+1}^k a_{i+1} E + \sum_{j=1}^k t_i^{k-j} t_{i+1}^j E$
\end{enumerate}
\end{lemma}
\begin{proof}

We prove (1) by induction on $k \geq 1$. We first assume $k = 1$. We have $a_i t_i = a_i^2 t_{i-1}a_i = t_{i-1}a_i + (q-1)
(a_i+1) e_{a_i} t_{i-1}a_i = t_{i-1}a_i + (q-1)e_{a_i} t_{i-1}a_i + (q-1) a_i e_{a_i} t_{i-1} a_i
= t_{i-1}a_i + (q-1)t_{i-1}a_ie_{a_i^{-1}t_{i-1}^{-1}a_it_{i-1}a_i} + (q-1) a_i  t_{i-1} a_i e_{t_{i-1}^{-1}a_it_{i-1}}
= t_{i-1}a_i + (q-1)t_{i-1}a_ie_{a_i^{-1}t_{i-1}^{-1}a_it_{i-1}a_i}+ (q-1) t_i e_{t_{i-1}^{-1}a_it_{i-1}}
\in t_{i-1}a_i E +t_i E$.
Now, if $k \geq 2$, then $a_i t_i^k=a_i t_i^{k-1}t_i \in 
t_{i-1}^{k-1} a_i Et_i + \sum_{j=1}^{k-1} t_{i-1}^{j-1} t_i^{k-j} Et_i
\subset t_{i-1}^{k-1} (a_i t_i) E + \sum_{j=1}^{k-1} t_{i-1}^{j-1} t_i^{k-j}t_i E
\subset t_{i-1}^{k-1} (t_{i-1}a_i E +t_i E) E + \sum_{j=1}^{k-1} t_{i-1}^{j-1} t_i^{k+1-j} E
\subset t_{i-1}^{k} a_i E +t_{i-1}^{k-1} t_i E + \sum_{j=1}^{k-1} t_{i-1}^{j-1}j t_i^{k+1-j} E
\subset t_{i-1}^{k} a_i E + \sum_{j=1}^{k} t_{i-1}^{j-1} t_i^{k+1-j} E$ and this proves (1).

We now prove (2) by induction on $k \geq 1$. If $k=1$, then $a_{i+1} t_i a_{i+1} = t_{i+1}$
implies $a_{i+1} t_i = t_{i+1} a_{i+1}^{-1} \in t_{i+1} a_{i+1} E + t_{i+1} E$. If
$k \geq 2$, then
$a_{i+1} t_i^k = a_{i+1} t_i^{k-1} t_i 
\in t_{i+1}^{k-1} a_{i+1} E t_i +  \sum_{j=1}^{k-1} t_i^{k-1-j} t_{i+1}^j E t_i
\subset t_{i+1}^{k-1} (a_{i+1}t_i) E  +  \sum_{j=1}^{k-1} t_i^{k-1-j} t_{i+1}^j t_iE 
\subset t_{i+1}^{k-1} (t_{i+1} a_{i+1} E + t_{i+1} E) E  +  \sum_{j=1}^{k-1} t_i^{k-j} t_{i+1}^j E 
\subset  t_{i+1}^{k-1}t_{i+1} a_{i+1} E + t_{i+1}^{k-1}t_{i+1} E  +  \sum_{j=1}^{k-1} t_i^{k-j} t_{i+1}^j E 
\subset  t_{i+1}^{k} a_{i+1} E   +  \sum_{j=1}^{k} t_i^{k-j} t_{i+1}^j E 
$ and this proves (2).
\end{proof}

Since $a_2,\dots,a_n$ satisfy the braid relations in type $A_{n-1}$, by Iwahori-Matsumoto theorem
we know that, for each $g \in \mathfrak{S}_n$
there is a well-defined $a_g \in B$ such that $a_g = a_{i_1}\dots a_{i_r}$ for every
reduced decomposition $g = s_{i_1}\dots s_{i_r}$ with $s_{i_m} = (m, m-1)$.
We note that, for each $i \geq 2$, $a_i a_g \in \sum_{h \in \mathfrak{S}_n} a_h E$, as a consequence
of the corresponding inequality inside $\mathcal{C}_{\mathfrak{S}_n}(u)$.
From this we prove that
$$
\mathcal{C}_W^p = \sum_{g \in \mathfrak{S}_n} \sum_{0 \leq k_1,\dots,k_n \leq d} t_1^{k_1}\dots t_n^{k_n} a_{g} E
$$
Indeed, the RHS contains $1$ and is clearly stable by left multiplication under
\begin{itemize}
\item $a_1 = t = t_1$, by the order relation $t^d = 1 + (q-1) P(t)e_t$
\item $a_2,\dots,a_n$ by lemma \ref{lem:techAK} and the fact that $a_i a_g E \subset \sum_{h \in \mathfrak{S}_n} a_h E$ for all $i \geq 2$.
\end{itemize}
Since $E$ is spanned by $|\mathcal{W}_p|$ elements, and $|W| = d^n n!$, this
proves that the assumption of proposition \ref{prop:crit1freeness} is satisfied, and this proves conjecture \ref{conj:freeness} for $W = G(d,1,n)$.

\end{document}